\documentclass[11pt]{article}
\usepackage{sustyle}
\usepackage[round]{natbib}
\graphicspath{{figs/}}
\usepackage{fullpage}

\newcommand\lasso{\widehat{\bbeta}}
\newcommand\wS{\widehat{S}}
\newcommand\fsset{F}

\newcommand{\pr}{\operatorname{P}}

\definecolor{wjs}{RGB}{0,0,255}

\newcommand{\bzero}{\bm{0}}
\newcommand{\bV}{\bm{V}}
\newcommand{\bD}{\bm{D}}

\newcommand{\C}{\mathcal{C}}
\newcommand{\tr}{\operatorname{\textrm{tr}}}

\title{When Is the First Spurious Variable Selected by Sequential Regression Procedures?}
\author{Weijie J.~Su}
\date{}

\begin{document}
\maketitle

{\centering
\vspace*{-0.5cm} 
Department of Statistics, University of Pennsylvania, Philadelphia, PA 19104, USA
\par\bigskip 
\date{July 11, 2018}\par
}

\begin{abstract}
Applied statisticians use sequential regression procedures to produce a ranking of explanatory variables and, in settings of low correlations between variables and strong true effect sizes, expect that variables at the very top of this ranking are truly relevant to the response. In a regime of certain sparsity levels, however, three examples of sequential procedures---forward stepwise, the lasso, and least angle regression---are shown to include the first spurious variable unexpectedly early. We derive a rigorous, sharp prediction of the rank of the first spurious variable for these three procedures, demonstrating that the first spurious variable occurs earlier and earlier as the regression coefficients become denser. This counterintuitive phenomenon persists for statistically independent Gaussian random designs and an arbitrarily large magnitude of the true effects. We gain a better understanding of the phenomenon by identifying the underlying cause and then leverage the insights to introduce a simple visualization tool termed the ``double-ranking diagram'' to improve on sequential methods.

As a byproduct of these findings, we obtain the first provable result certifying the exact equivalence between the lasso and least angle regression in the early stages of solution paths beyond orthogonal designs. This equivalence can seamlessly carry over many important model selection results concerning the lasso to least angle regression.

\end{abstract}

\section{Introduction}\label{sec:introduction}

Consider observing an $n$-dimensional response vector $\by$ that is generated by a linear model
\[
\by = \bX \bbeta+ \bz,
\]
where $\bX \in \R^{n \times p}$ is a design matrix, $\bbeta \in \R^p$ is a vector of regression coefficients, and $\bz \in \R^n$ is a noise term. To find explanatory variables that are associated with the response $\by$, especially in the setting where $p > n$, three sequential regression procedures are frequently used: forward stepwise regression,  the lasso \citep{tibshirani1996regression}, and least angle regression \citep{efron2004least}. These popular methods build a model by sequentially adding or removing variables based upon some criterion. In a very natural way, a sequential method ranks explanatory variables according to when the variables enter the solution path. With this ranking of variables in place, a routine practice for forming the final model is to select all variables ranked earlier than a certain cutoff and discard the rest.

When running a sequential procedure, a practitioner often wishes to understand where along the solution path noise variables (regressors with zero regression coefficients) start to enter the model. In particular, when is the first noise variable selected? A better understanding of this problem is desirable from at least two perspectives. First, the rank of the first noise variable sheds light on the difficulty of consistent model selection, offering guidelines for selecting important variables. More precisely, if the rank is about the same size as the sparsity (the total number of nonzero regression coefficients), we could obtain a model retaining most of the important variables without causing many false selections using a sequential method, whereas a small rank implies that signal variables (regressors with nonzero regression coefficients) and noise variables are interspersed early on in the solution path and, as a result, a false selection must occur long before the power reaches one. Second, the empirical performance of numerous tools for post-selection inference in linear regression is, to a large extent, contingent upon whether the first noise variable occurs early or not \citep{lockhart2014significance,g2016sequential,tibshirani2016exact}. Insights into the occurrence of the first false variable would be valuable for improving these tools and developing new ones.

However, despite an extensive body of work on these sequential methods, the literature remains relatively silent on questions of the first false variable. Existing results address these questions in a limited setting, mostly characterizing under what conditions all the signal variables precede the first noise variable, that is, perfect support recovery or, put more simply, selecting the exactly correct model. Specifically, this set of results guarantees perfect support recovery using a certain sequential method provided sufficiently strong effect sizes compared to the noise level and a form of local orthogonality of the design matrix. These results can be found for both the lasso \citep{zhao2006model,bickel2009simultaneous,wainwright2009sharp} and forward stepwise \citep{tropp2004greed,zhang2009consistency,cai2011orthogonal}.

Figure \ref{fig:intro} illustrates a simulation study that examines when the first noise variable gets selected by the lasso. The design matrix $\bX$ is of size $2000 \times 1800$ consisting of independent $\N(0, 1/2000)$ entries, the noise term $\bz$ is comprised of independent standard normals, and the regression coefficients are set to $\beta_1  = \cdots = \beta_k = 100\sqrt{2\log(1800)} = 387.2$ and $\beta_j = 0$ for all $j > k$, with the sparsity $k$ varying from $10$ to $320$. Note that the true effect sizes can be practically thought of as infinitely strong and the sample correlations between the regressors are small due to the independence. Fig.~\ref{fig:intro} shows that, in the low sparsity regime, the pairs (sparsity, rank) lie close to the $45^{\circ}$ line
(precisely, it is the line $y = x + 1$). This behavior is equivalent to saying that all the signal variables are selected prior to any false variables, which is in perfect agreement with a copious body of theoretical results available in the literature.

\begin{figure}[h!]
\centering
\includegraphics[scale=0.4]{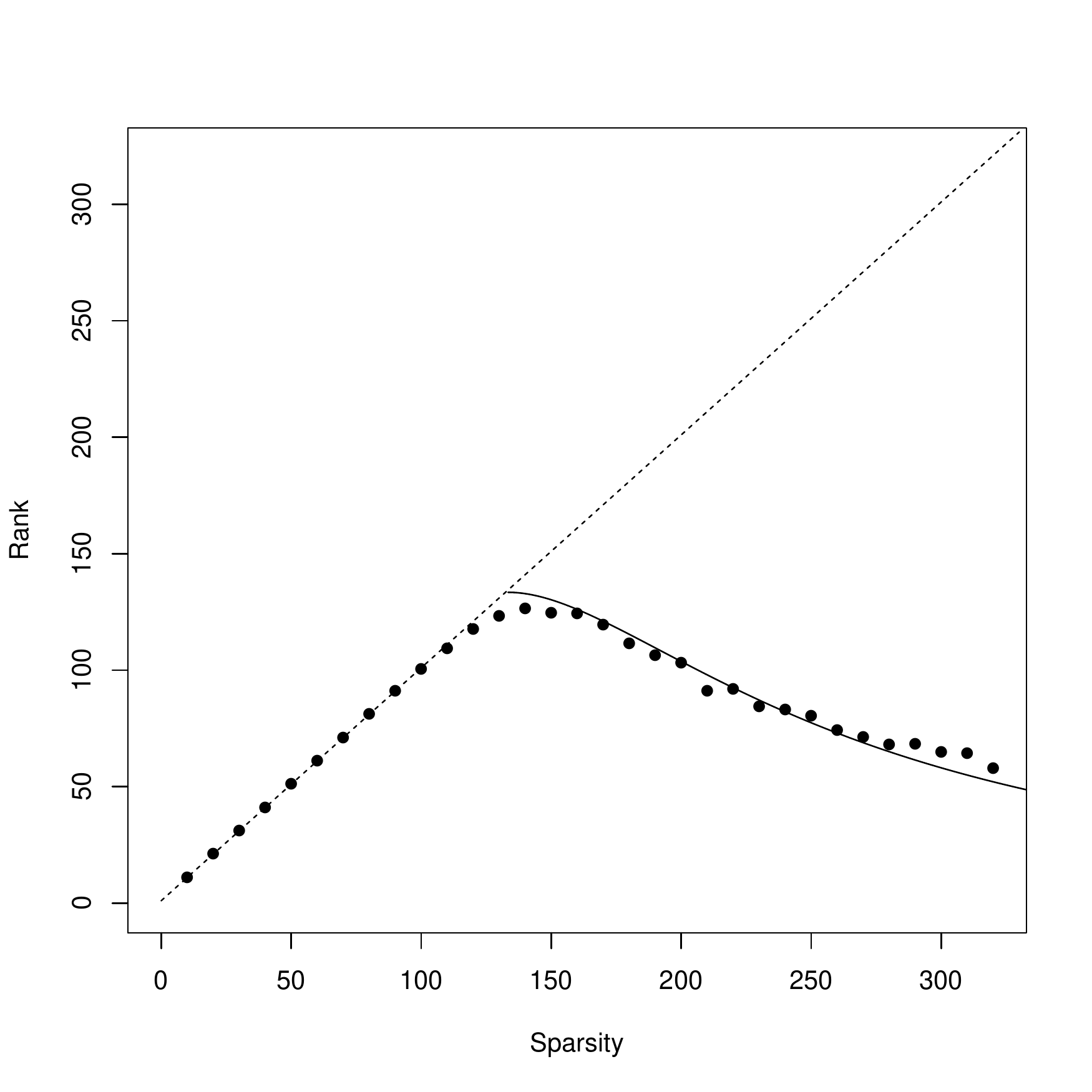}
\caption{Rank of the first spurious variable along the lasso path. Recall that the rank equals one plus the number of signal variables preceding the first spurious variable. We plot averages from 500 independent
  replicates as dots. The $45^\circ$ dashed line is shown for
comparison. The solid line plots $\exp(\sqrt{(2n\log p)/k} - n/(2k) + \log(n/(2 p \log p)))$ as a function of $k$, starting from $n/(2\log p)$.}
\label{fig:intro}
\end{figure}

Strikingly, once the sparsity exceeds a certain level (around $140$ in the example), a phenomenon that is not explained by existing theory occurs: the average rank of the first noise variable becomes substantially smaller than the sparsity $k$ and, more surprisingly, the rank keeps decreasing as the sparsity increases. This phenomenon clearly demonstrates the impossibility of perfect support recovery in this non-extreme sparsity regime using the lasso, even though it is under high signal-to-noise ratios and low correlations. Presumably, as the signal $\bbeta$ is amplified by setting more components to a large magnitude, one might instinctively anticipate that a sequential method such as the lasso tends to include more signal variables at the beginning and, thus, would imagine that
the first noise variable would get selected later and later. Unfortunately, the counterintuitive results as shown in Fig.~\ref{fig:intro} falsify this belief. We remark that a similar phenomenon is observed earlier in \citet{su2015false}, although they did not provide any justification for the observation.

Thus, concrete predictions and explanations are needed to better understand and improve sequential methods in this sparsity regime. In response, we derive an analytical prediction that is asymptotically \textit{exact} for the first noise variable. The prediction applies to the three methods under our consideration, namely forward stepwise, the lasso, and least angle regression, and potentially to other sequential methods. Denote by $T$ the rank of the first noise variable. Informally, the prediction states that, in the setting of strong effect sizes and statistically independent regressors as in Fig.~\ref{fig:intro}, the three sequential procedures in the non-extreme sparsity regime all satisfy
\begin{equation}\label{eq:t_pred_intro}
\log T \approx  \sqrt{\frac{2n\log p}{k}} - \frac{n}{2k} + \log\frac{n}{2 p \log p}.
\end{equation}
The formal statement of this result is given in Theorem \ref{thm:lower} in \S~\ref{sec:theory-underst-phen}. 

The prediction of $T$ is additionally presented in Fig.~\ref{fig:intro}, showing excellent agreement between the predicted and observed behaviors. To better appreciate this result, note that the quantity as an approximation to $\log T$ in \eqref{eq:t_pred_intro} is smaller than $\log k$ once the sparsity $k$ exceeds $n/(2\log p)$, suggesting the impossibility of perfect support recovery in this regime. This is consistent with the negative result in Corollary 2 of \citet{wainwright2009sharp}. The prediction \eqref{eq:t_pred_intro}, however, implies more. To show this, alternatively write the right-hand side of \eqref{eq:t_pred_intro} as
\[
\sqrt{\frac{2n\log p}{k}} - \frac{n}{2k} + \log\frac{n}{2 p \log p}= - \left[  \sqrt{\log p} - \sqrt{\frac{n}{2k}}\right]^2 + \log \frac{n}{2\log p}.
\]
The expression above reveals that the predicted $\log T$ decreases as the sparsity $k \ge n/(2\log p)$ increases. Put differently, the first noise variable is bound to occur earlier as the signal vector $\bbeta$ gets denser, successfully predicting the phenomenon shown in Fig.~\ref{fig:intro}. While problems in selecting the true model by the lasso have been empirically documented in earlier work \citep{fan2010sure}, such sharp and analytical predictions are not available in the literature, perhaps due to technical difficulties.

This result has several implications. First, once the underlying signals go beyond the very sparse regime, using sequential procedures would inevitably lead to a very low power with familywise error rate control, which is the probability of selecting one or more noise variable, no matter how large the effect sizes are. Taking a simple example in which both $n$ and $p$ are set to be equal and large and $k = \epsilon p$ for some fixed $0 < \epsilon < 1$, the prediction asserts that the first false variable is included after no more than
\[
\exp\left[(1+o(1))\left(\sqrt{2(\log p)/\epsilon} - 1/(2\epsilon) - \log(2\log p)\right)\right] = \exp\left[(1+o(1))\sqrt{2(\log p)/\epsilon}\right]
\] steps (note that $\sqrt{2\log p/\epsilon}$ is the leading component for a large $p$). For a fixed $\epsilon$, however, the predicted rank $\exp\left[(1+o(1))\sqrt{2\log p/\epsilon}\right]$ only accounts for a vanishing fraction of the $k = \epsilon p$ signal variables, which can be gleaned from the fact that $\sqrt{2\log p/\epsilon} = o(\log p)$. In other words, the three sequential methods being considered yield vanishing power if no noise variable is allowed to be included, even in the noiseless case ($\bz = \bzero$). In particular, these negative results are derived under Gaussian designs with independent columns, which have vanishing sample correlations and satisfy some conditions believed to
be favorable for model selection, including restricted isometry properties \citep{CanTaoDecode05} and restricted eigenvalue conditions \citep{bickel2009simultaneous}. Thus, the negative results are likely to carry over to a much broader class of design matrices. In fact, extensive simulations carried out in \S~\ref{sec:examples} demonstrate that problems of the first false variable are only exacerbated in more general settings.

Another implication yielded by this prediction is that the three sequential regression methods seem to behave similarly in ranking variables, at least in the independent random design setting. Compared with forward stepwise, the lasso and least angle regression, along with their infinitesimal version forward stagewise regression (see, for example, \citet{efron2004least}), are long-time considered less greedy because at each step they gradually blend in a new variable instead of adding it discontinuously \citep{efron2004least}. To be more precise, the forward stepwise selects the predictor with the largest absolute correlation with the residual vector and then aggressively takes a large step in the direction of the selected predictor, whereas the others proceed in a more democratic
manner along a direction equiangular between the set of selected predictors (the lasso and forward stagewise regression bear certain restrictions on this equiangular approach). This critical distinction between the two strategies is anticipated---or, at least wished---to lead to contrasting model selection performance. Interestingly, this is \textit{not} the case; these two strategies yield the same behavior of selecting the first noise variables in our setting. As a byproduct, we obtain Theorem \ref{thm:unique} for the lasso and least angle regression, which, to the best of our knowledge, is the first mathematically provable result certifying the exact equivalence between early solution paths of these two procedures beyond orthogonal designs.

In the non-extreme sparsity regime, why do these distinct sequential methods select the first false variable so early? Taking a closer look at the derivation of the prediction, we can identify the cause, which is, loosely speaking, due to the greedy nature of these sequential regression methods. Moreover, the equiangular strategy adopted by the lasso and least angle regression fails to alleviate greediness from the perspective of when the first noise variable gets selected. To shed light on this cause, recall that all the three methods at each time include a variable that roughly has the largest absolute inner product with the current residuals. As the regression coefficients get denser, the solution at the beginning of the path is overwhelmingly biased and the residual vector absorbs many of the true effects contributed by
the nonzero components of $\bbeta$. As a result, some irrelevant variable would exhibit high correlations with the residuals and hence is selected incorrectly and early. That being said, it requires several novel ideas to precisely characterize what we describe here.

With this underlying cause in mind and to improve on sequential methods, we introduce the double-ranking diagram to identify early false variables along solution paths. In slightly more detail, this diagram contrasts the rank of each variable given by a sequential procedure (horizontal axis) with that given by a low-bias estimator (vertical axis) such as the least-squares estimator. In spite of a significant horizontal rank, an early noise variable might be revealed by its possibly less significant vertical rank. Related ideas have appeared in recent variable screening work, for instance, \citet{wang2016high}. We demonstrate the usefulness of this diagram via a mix of theoretical and empirical results.


\section{Understanding the Phenomenon}
\label{sec:theory-underst-phen}

\subsection{Predicting the first spurious variable}
\label{sec:pred-first-noise}
We consider a sequence of problems indexed by $(k_l, n_l, p_l)$, where $k_l, n_l$, and $p_l$ are all assumed to grow to infinity as $l \goto \infty$ in asymptotic statements. The subscript $l$ is often omitted when clear from the context. Letters $c_i$ and $C_i$ in various settings denote positive constants that do not depend on the problem index $l$. Below we formalize our working hypothesis concerning the linear model $\by = \bX\bbeta + \bz$.

\begin{assumption}\label{ass:working}
The design $\bX \in \R^{n \times p}$ has independent $\N(0, 1/n)$ entries and $\bz \in \R^n$ consists of independent $\N(0, \sigma^2)$ errors. We further assume $\bX$ and $\bz$ are independent. The coefficient vector $\bbeta$ has $k$ fixed components equal to some $M \ne 0$ and the rest are all zero. Last, we assume $c_1p/\log^{c_2} p \le n \le c_3 p$ and $c_4 n \le k \le \min\{0.99 p, c_5 n \log^{0.99} p\}$ for arbitrary positive constants $c_1, c_2, c_3, c_4$, and $c_5$.
\end{assumption}

The assumption on $(k, n, p)$ is satisfied in some popular examples studied in the literature, for instance, the linear sparsity framework where $k/p$ and $n/p$ converge to some constants \citep{bayati2012lasso}. Moreover, a number of cases leading to $k = o(p)$ satisfy Assumption \ref{ass:working}, for instance, $n = c_1 p/\log^{c_2} p$ and $k = c_4 n$. Under this assumption, each column of $\bX$ is approximately normalized, having about unit Euclidean norm. This random design is conventionally considered to be easy for model selection since it obeys restricted isometry properties \citep{CanTaoDecode05} or restricted eigenvalue conditions \citep{bickel2009simultaneous} with high probability. The nonrandom
parameters $\sigma \ge 0$ and $M$ both implicitly depend on the index $l$ and thus are allowed to vary freely. In particular, the noiseless case $\sigma = 0$ is not excluded, in which the signal-to-noise ratio is essentially infinite.  For completeness, the number $0.99$ can be replaced by any positive number smaller than 1.

Before presenting our main results Theorems \ref{thm:upper} and \ref{thm:lower}, we give a brief overview of the three methods for ease of reading. In broad outlines, least angle regression increases the coefficients of included variables in their joint least squares direction until an unselected variable has as much inner product with the residuals, which is included in the next step. Least angle regression stops when the residuals are zero or all variables are included. If a nonzero coefficient is removed from the active set whenever it hits zero, this adjustment leads to the lasso, which is better-known as the minimizer of the convex program $\frac12 \|\by - \bX \bb\|_2^2 + \lambda \| \bb\|_1$ over $\bb \in \R^p$ for $\lambda$ ranging from infinity to zero. Forward stepwise is described in detail in \S~8.5 of \citet{weisberg1980applied}. In our setting, the intercept is not included and normalization is not applied to any columns of $\bX$. Last, recall that $T$ denotes the rank of the first noise variable, and $o_{\P}(1)$ denotes a sequence of random variables converging to zero in probability.
\begin{theorem}\label{thm:upper}
Under Assumption \ref{ass:working}, the first spurious variable selected by each of forward stepwise, the lasso, and least angle regression satisfies
\[
\log T\le (1+o_{\P}(1)) \left[ \sqrt{2n(\log p)/k} - n/(2k) + \log(n/(2p\log p))\right].
\]
\end{theorem}

If the signal magnitude $M$ is not sufficiently large compared with $\sigma$, the logarithm of $T$ would be much smaller than the upper bound appearing in the display above, meaning that the problems of the first false variables could be worse. Interestingly, this bound is sharp when $M$ is sufficiently large compared with $\sigma$, as demonstrated in the theorem below.

\begin{theorem}\label{thm:lower}
Under Assumption \ref{ass:working} and in addition provided that $\sigma/M \goto 0$, the three sequential methods obey
\[
\log T = (1+o_{\P}(1)) \left[ \sqrt{2n(\log p)/k} - n/(2k) + \log(n/(2p\log p))\right].
\]
\end{theorem}

Provided in the Appendix, the proofs of both theorems involve some techniques that are likely to extend beyond the three sequential methods. The condition concerning the ratio between $\sigma$ and $M$ can be relaxed to $|M|/\sigma \gg \sqrt{n/k}$. Setting $k \approx n \log^{0.99} p$ as in Assumption \ref{ass:working}, for example, Theorem \ref{thm:lower} follows if $M/\sigma$ is bounded away from $0$. An immediate consequence of this theorem is as follows.

\begin{corollary}\label{corr:noiseless}
Under Assumption \ref{ass:working}, each of the three methods in the noiseless case ($\sigma = 0$) obeys
\[
\log T = (1+o_{\P}(1)) \left[ \sqrt{2n(\log p)/k} - n/(2k) + \log(n/(2p\log p))\right].
\]
\end{corollary}

In addition to predicting the phenomenon observed in Fig.~\ref{fig:intro}, Theorem \ref{thm:lower} together with Corollary \ref{corr:noiseless} demonstrates that having an even stronger signal magnitude does not affect $T$ much as long as it exceeds a certain level. 

The theorems presented here differ from results that are found extensively in the literature claiming a high probability of selecting the exactly correct model, mainly due to assuming different sparsity regimes of the regression coefficients $\bbeta$. Explicitly, the former assumes $c_4 n \le k \le \min\{0.99 p, c_5 n \log^{0.99} p\}$ whereas the latter often, if not always, assumes a restrictive sparsity regime such as $k = O(n/\log p)$ or $k \ll n/\log p$. In fact, under Assumption \ref{ass:working}, it is unrealistic to expect perfect model selection using sequential methods: below a simple corollary of Theorem \ref{thm:upper} shows the number of signal variables before the first false variable only accounts for an insignificant fraction of the total number of signal variables.

\begin{corollary}\label{corr:vanish}
Under Assumption \ref{ass:working}, each of the three methods satisfies
\[
\frac{T}{k} \longrightarrow 0 \text{ in probability}.
\]
\end{corollary}

To better appreciate Corollary \ref{corr:vanish}, consider the scenario where $k/p \goto \epsilon$ and $n/p \goto \delta$ for some positive constants $\epsilon < 1$ and $\delta$. Theorem \ref{thm:upper} shows that, up to a vanishing fraction, the logarithm of $T$ is no larger than $\sqrt{2\delta(\log p)/\epsilon} - \delta/(2\epsilon) + \log(\delta/(2\log p)) = (1+o(1))\sqrt{2\delta(\log p)/\epsilon}$. This expression for approximating $\log T$ yields $T \le \exp((1+o(1))\sqrt{2\delta(\log p)/\epsilon}) \ll \epsilon p = k$, confirming Corollary \ref{corr:vanish} in this linear sparsity regime. We summarize the finding in the corollary below.

\begin{corollary}
Under Assumption \ref{ass:working} and additionally provided that $k/p \goto \epsilon$ and $n/p \goto \delta$ for arbitrary positive constants $\epsilon < 1$ and $\delta$, each of the three methods satisfies
\[
T \le \e^{(1+o_{\P}(1))\sqrt{2\delta(\log p)/\epsilon}}.
\]
\end{corollary}

This regime of linear sparsity is previously employed in \cite{su2015false}, which studies limitations of the lasso for the false discovery rate control. The techniques developed there are not applicable to  studying the first noise variable, which is a much finer problem.

\subsection{Equivalence between lasso and least angle regression}
\label{sec:equiv-betw-lasso}

In contrast to the other two methods, the lasso would drop a selected variable if its coefficient hits zero. This irregularity of the lasso path might lead to ambiguity in interpreting the rank $T$ in Theorems \ref{thm:upper} and \ref{thm:lower}. Fortunately, as a byproduct of the above, the theorem below rules out the possibility of such ambiguity.

\begin{theorem}\label{thm:unique}
Assume $\bX$ has independent $\N(0, 1/n)$ entries. Then, with probability at least $1 - p^{-2}$, no drop-out occurs before the first
\[
\min\left\{ \left\lceil c\sqrt{n/\log p} \right\rceil, p \right\}
\]
variables along the lasso path are selected, where $\lceil x \rceil$ denotes the least integer greater than or equal to $x$ and $c > 0$ is some universal constant.
\end{theorem}

Note that Theorem \ref{thm:unique} only requires the normality of $\bX$, as opposed to additional conditions imposed on $\bz, \bbeta$, and $(k, n, p)$ in Assumption \ref{ass:working}. As seen from its proof in the Appendix, the validity of the theorem depends on the design matrix $\bX$ basically only through its restricted isometry property. Thus, this result can seamlessly carry over to other matrix ensembles with an appropriate restricted isometry property constant, such as Bernoulli random matrices \citep{CanTaoDecode05}.

Under Assumption \ref{ass:working}, $\log(\min\{\lceil c\sqrt{n/\log p}\rceil, p\}) = \log \lceil c\sqrt{n/\log p}\rceil \gg \sqrt{2n(\log p)/k} - n/(2k)$. Consequently, Theorem \ref{thm:unique} together with Theorem \ref{thm:upper} ensures that the first noise variable selected by the lasso is not preceded by any drop-out with probability approaching one.

This byproduct provides new insights into the lasso path and is a contribution of independent interest to high-dimensional statistics. The lasso is known to coincide exactly with least angle regression until the first time the lasso drops a selected variable \citep{efron2004least,tibshirani2011solution}. To our knowledge, however, the question of where along the path the lasso and least angle regression differ has not been addressed in prior research, perhaps due to technical difficulties. By confirming the equivalence between these two procedures, Theorem \ref{thm:unique} allows us to carry over well-known results on the lasso for model selection to least angle regression.

\subsection{Heuristics and insights}
\label{sec:heur-expl}

In this section we give an informal derivation of Theorems \ref{thm:upper} and \ref{thm:lower}. Although our discussion below lacks rigor, nevertheless, the goal is to gain insights into this counterintuitive phenomenon. For the full proofs, see the Appendix.

We focus on the noiseless case $\bz = \bzero$ in Assumption \ref{ass:working}, which is presumably the most ideal scenario for model selection. Denote by $S = \{j: \beta_j \ne 0\}$ the support of the signals and by $\widehat\bbeta$ an estimate given by any of the three sequential methods somewhere along the solution path. Write $j_1 \notin S$ for the index off the support having the largest inner product in magnitude with the residual $\by - \bX\widehat\bbeta = \bX(\bbeta - \widehat\bbeta)$, and $j_2 \in S$ for the index on the support having the $(T-1)$th largest inner product in magnitude with the residual. By using some technical arguments found in the Appendix, we get
\begin{equation}\label{eq:corr}
\bX_{j_1}^\top \bX(\bbeta - \widehat\bbeta) \approx M \sqrt{\frac{2k\log (p-k)}{n}}, \quad \bX_{j_2}^\top \bX(\bbeta - \widehat\bbeta) \approx M + M \sqrt{\frac{2k\log (k/T)}{n}}.
\end{equation}
Above and henceforth, $\bX^\top$ denotes the transpose of
$\bX$. Recognizing that the sequential methods rank variables essentially according to the correlations with the residual, where in our case correlations are roughly equivalent to inner products since the columns of $\bX$ are approximately normalized, from \eqref{eq:corr} we must have
\[
M\sqrt{\frac{2k \log(p-k)}{n}} \approx M + M\sqrt{\frac{2 k \log(k/T)}{n}}
\]
at the point where the first false variable is just about to enter the model. In the linear sparsity regime $k/p \goto \epsilon, n/p \goto \delta$, this yields 
\[
\log T \approx \sqrt{\frac{2\delta \log p}{\epsilon}}.
\] 

The exposition above suggests that an early spurious variable is mainly due to a large inner product $\bX_{j_1}^\top \bX(\bbeta - \widehat\bbeta)$, which would not be the case if $\widehat\bbeta$ was a low-bias estimator of $\bbeta$. However, until a significant proportion of the variables have been selected, a solution $\widehat\bbeta$ provided by a sequential method is overwhelmingly biased. Another way to
formalize this point is that the residual $\bX(\bbeta - \widehat\bbeta)$ still contains a significant amount of true effects, largely contributed by presently unselected variables. This bias acts as if it were noise and, as a consequence, some irrelevant variables happen to correlate highly with the residual vector, leading to false variables selected early. This is not a matter of the signal-to-noise ratio; an increasing signal magnitude would enlarge the bias as well and, hence, noise variables always occur early. Other examples of \textit{pseudo} noise caused by bias have been observed in previous work
\citep{bayati2012lasso}. To be complete, we remark that this phenomenon does not appear in regimes of extreme sparsity (see, for example, \citet{wainwright2009sharp}).


\section{Illustrations}
\label{sec:examples}

\subsection{Numerical examples}
\label{sec:simulations}
We present simulation experiments to illustrate the first false variable of the three sequential methods, along with the predictions given by Theorems \ref{thm:upper} and \ref{thm:lower}. Specifically, we numerically examine three studies concerning the effect of design matrix shapes, signal magnitudes, and correlations between the columns of $\bX$ on the first spurious variable. Two scenarios are experimented for each study.

\par
\medskip
{\it Study} 1. In the first experiment (square design) the design $\bX$ of size $1000 \times 1000$ has independent $\N(0, 1/1000)$ entries, the signals $\beta_j = 100$ for $j \le k$ and $\beta_j = 0$ for $j \ge k+1$, and each noise component $z_i$ follows $\N(0, 1)$ independently. In the second experiment (fat design) the design $\bX$ is changed to be size of $800 \times 1200$ and has independent Bernoulli entries, which take value $1/\sqrt{500}$ with probability half and otherwise $-1/\sqrt{500}$, while all the other assumptions remain the same. Results of the two experiments are shown in Figure \ref{fig:study}(a) and (b), respectively.
\medskip

{\it Study} 2. In both experiments, the $500 \times 1000$ design matrix $\bX$ consists of independent $N(0, 1/500)$ entries and each $z_i$ is independently distributed as $\N(0, 1)$. For the first experiment (one mixture), we set $\beta_j = M$ for $j = 1, \ldots, 80$ and $\beta_j = 0$ for $j = 81, \ldots, 1000$. For the second one (two mixtures), we set $\beta_j = M$ for $j = 1, \ldots, 40, \beta_j = M^2/(10\sqrt{2\log p})$ for $j = 41, \ldots, 80$ and $\beta_j = 0$ for $j = 81, \ldots, 1000$. The parameter $M$ is varied from $0.2\sqrt{2\log p}$ to $10\sqrt{2\log p}$. Note that the two mixtures take the same value when $M = 10\sqrt{2\log p} = 37.17$. Results are shown in Fig.~\ref{fig:study}(c) and (d).
\par
\medskip

{\it Study} 3. This scenario uses $\bbeta$ obeying $\beta_j = 100\sqrt{2\log p}$ for $j \le 80$ and $\beta_j = 0$ otherwise. The noise $\bz$ consists of independent standard normals. The $500 \times 1000$ design matrix $\bX$ has each row independently drawn from $\N(\bzero, \bm\Sigma)$. For the $1000 \times 1000$ covariance matrix $\bm\Sigma$, the first experiment (equi correlation) assumes $\Sigma_{ij} = \rho/n$ if $i \ne j$ and $\Sigma_{jj} = 1/n$. In the second one (decaying correlation), $\Sigma_{ij} = \rho^{|i-j|}/n$. Results are shown in Fig.~\ref{fig:study}(e) and (f).
\par
\medskip

Both the lasso and least angle regression closely match our predictions. Notably, the two procedures yield exactly the same ranks of the first noise variables, hence supporting Theorem \ref{thm:unique}. On the other hand, forward stepwise exhibits larger departures from the theoretical predictions, mainly due to the slow convergence to the asymptotics, while as well showing a decreasing rank once the sparsity exceeds a cutoff.  

As shown in Fig.~\ref{fig:study}(a) and (b), the first false variable occurs earlier as $n$ decreases while $p$ gets larger. In particular, the behaviors of the methods under Bernoulli random designs as in Fig.~\ref{fig:study}(b) closely resemble that under Gaussian random designs. In Fig.~\ref{fig:study}(c), the rank of the first false variable increases as the signal magnitude $M$ is amplified. While this increasing rank is expected, Fig.~\ref{fig:study}(d) in contrast illustrates a rather surprising phenomenon: the rank drops after $M$ exceeds a certain level. More precisely, given $M \ge 3.4\sqrt{2\log p} = 12.64$, the lasso selects the first false variable earlier and earlier even though the sparsity is fixed and each signal gets strengthened, and the phenomenon is
even more transparent for forward stepwise. Intuitively, this is because the \textit{effective} sparsity in the case of a moderately large $M$ is smaller than the nominal sparsity $80$. To see this, observe that the ratio of the signals of the first 40 components and the next 40 components is $M/(M^2/37.17) = 37.17/M$, which is noticeably larger than 1. Put another way, the first 40 components act as the main signals and, hence loosely speaking, the effective sparsity is smaller than $80$. In the presence of significant correlations between columns of $\bX$, Fig.~\ref{fig:study}(e) and (f) clearly show that the problem of early false variables is further exacerbated.

\begin{figure}[ht!]
\centering
\hspace{0.01\linewidth} {\scriptsize (a)} \hspace{0.35\linewidth} {\scriptsize (b)} \hspace{0.48\linewidth}\\[-1.5em]
\includegraphics[width=0.4\textwidth, height=2.2in]{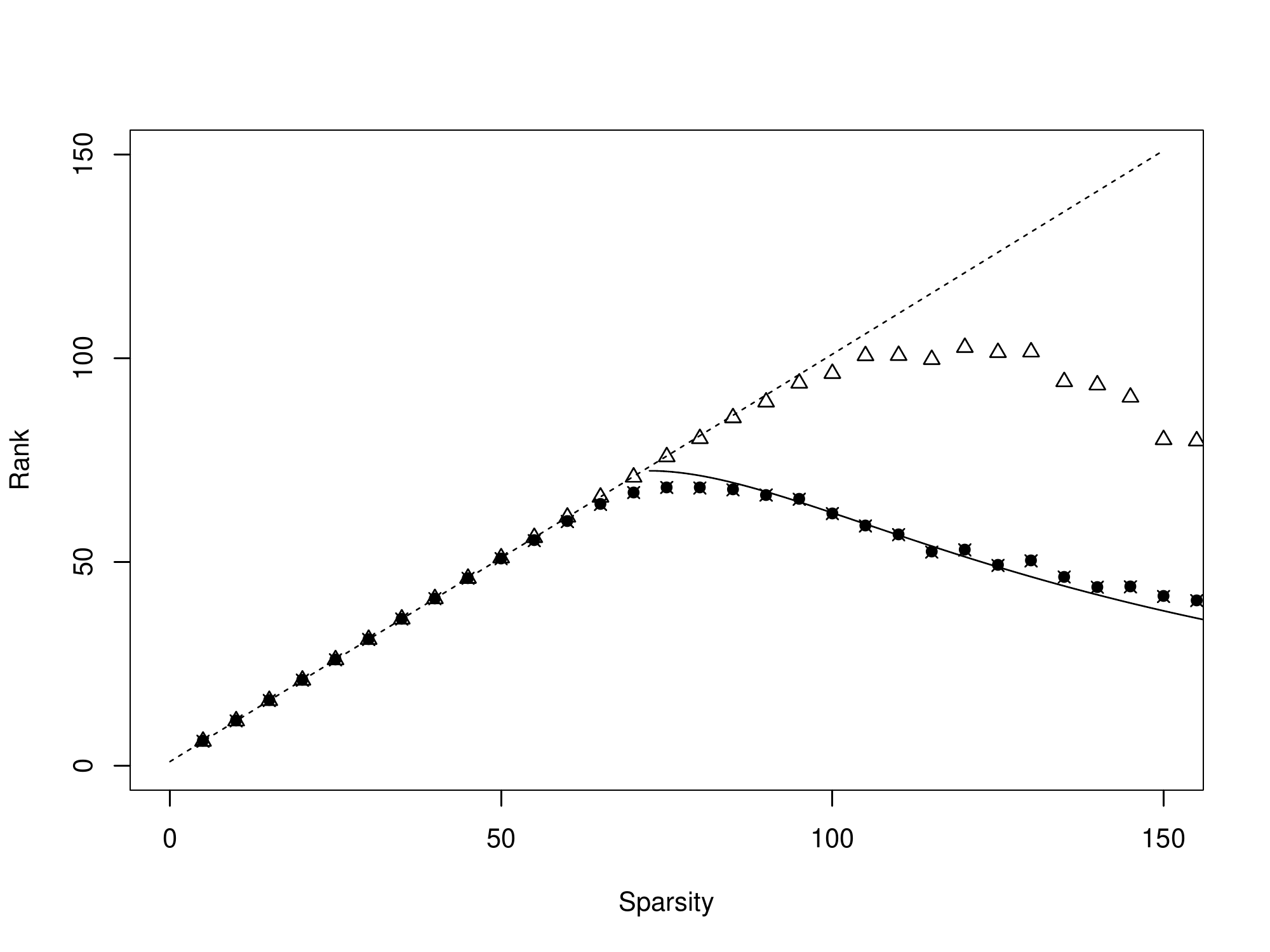}    
\includegraphics[width=0.4\textwidth, height=2.2in]{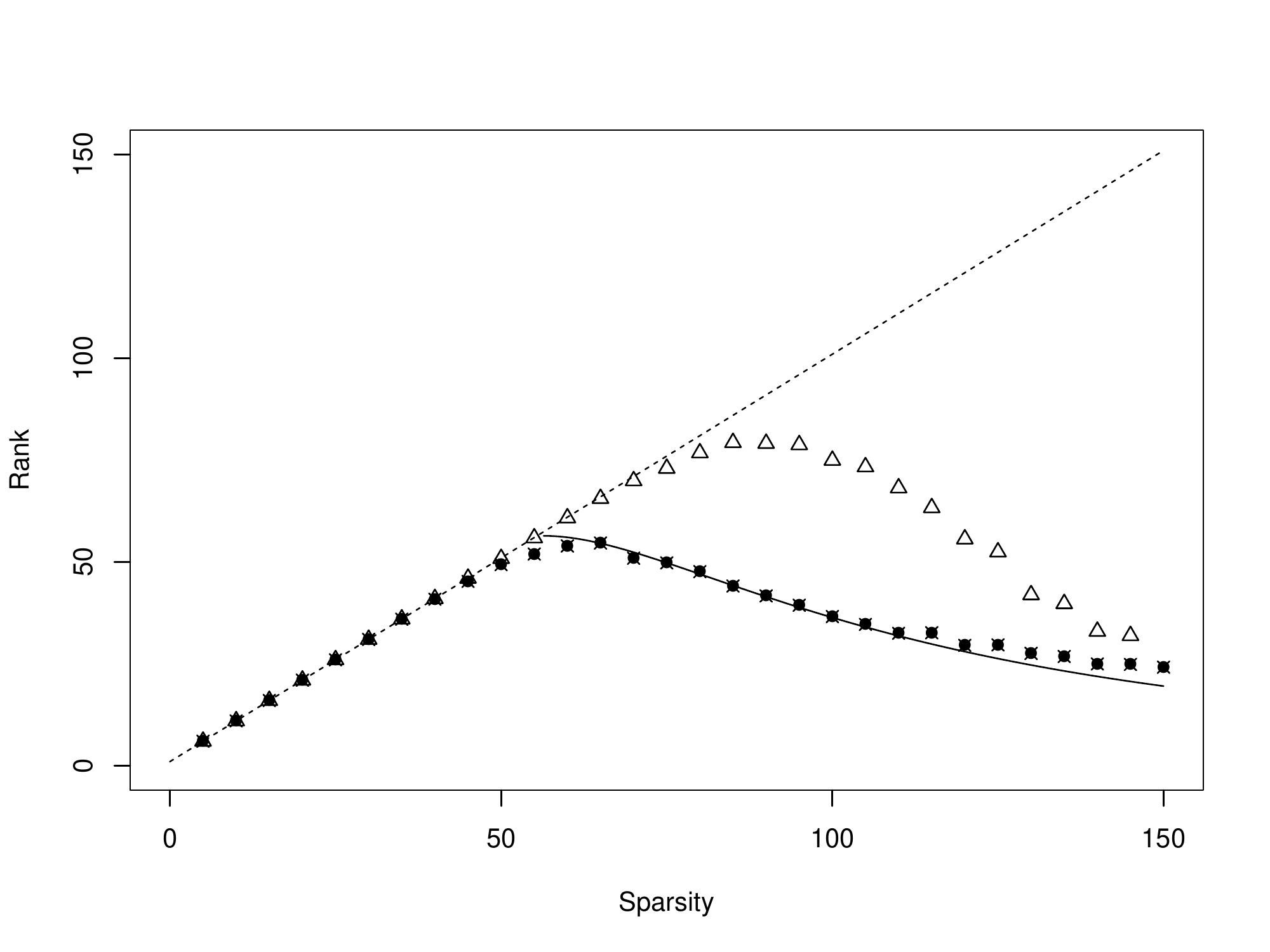}\\[-0.5em]
\hspace{0.01\linewidth} {\scriptsize (c)} \hspace{0.35\linewidth} {\scriptsize (d)} \hspace{0.48\linewidth}\\[-1.5em]
\includegraphics[width=0.4\textwidth, height=2.2in]{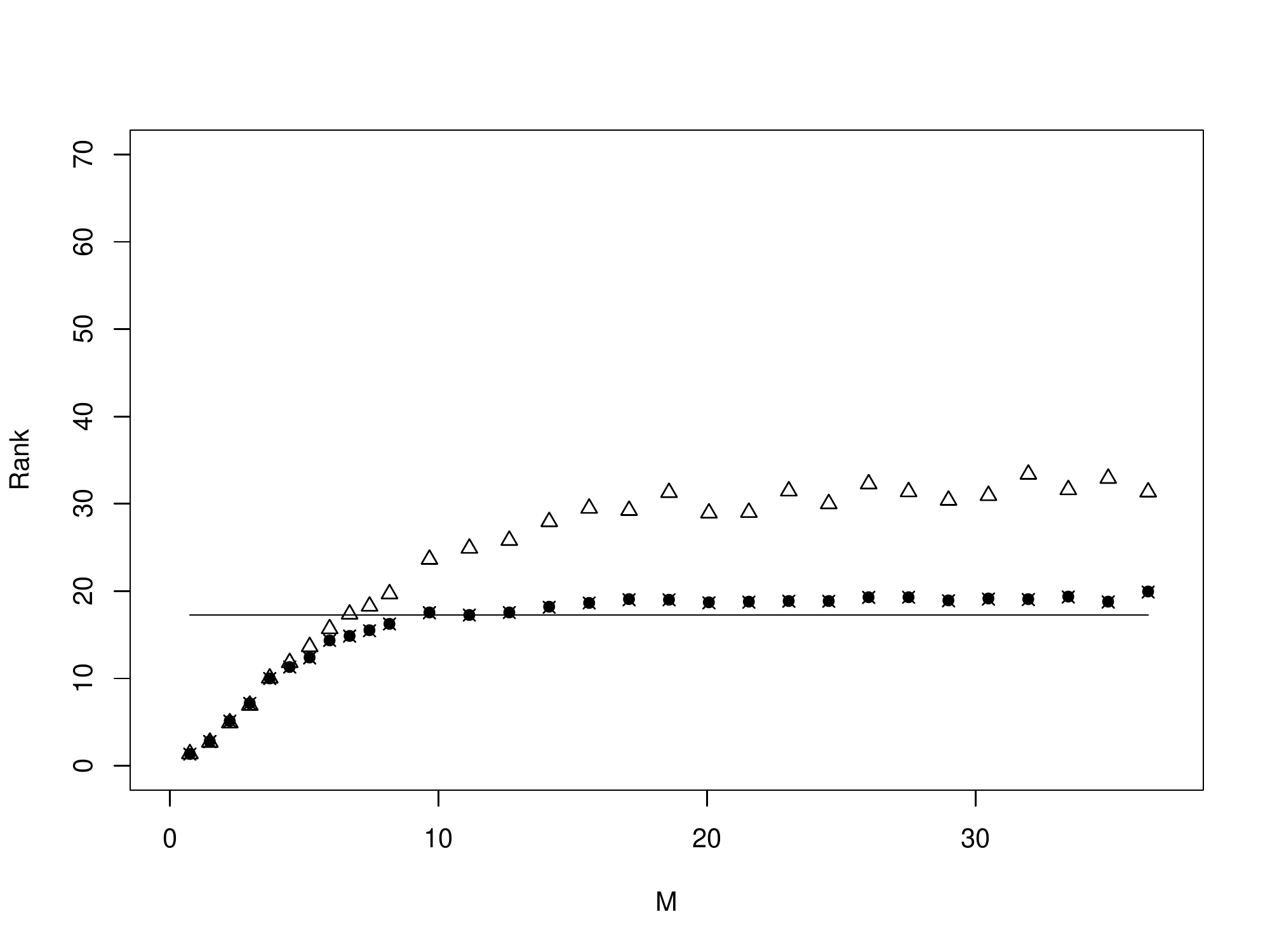}    
\includegraphics[width=0.4\textwidth, height=2.2in]{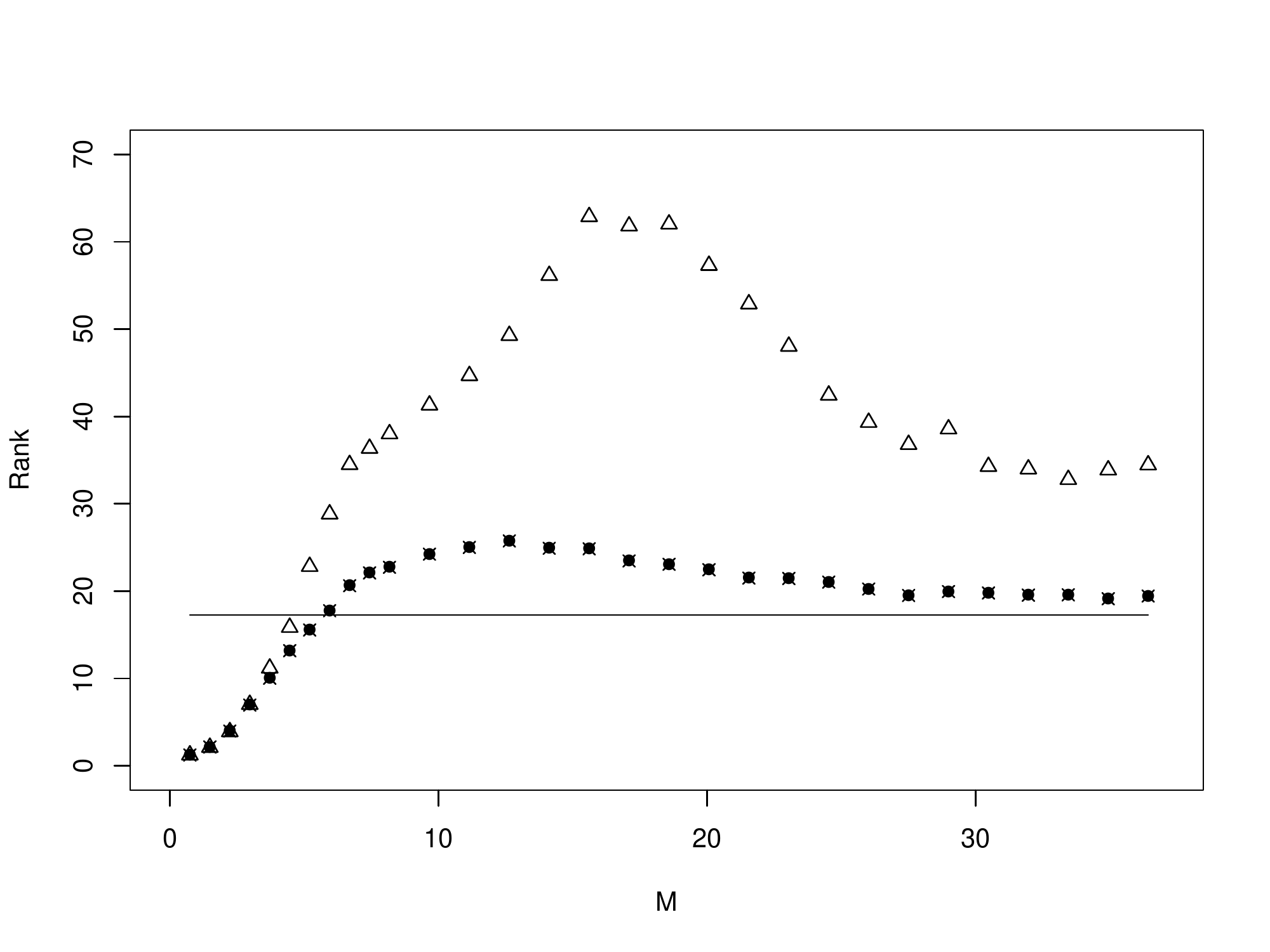}\\[-0.5em]
\hspace{0.01\linewidth} {\scriptsize (e)} \hspace{0.35\linewidth} {\scriptsize (f)} \hspace{0.48\linewidth}\\[-1.5em]
\includegraphics[width=0.4\textwidth, height=2.2in]{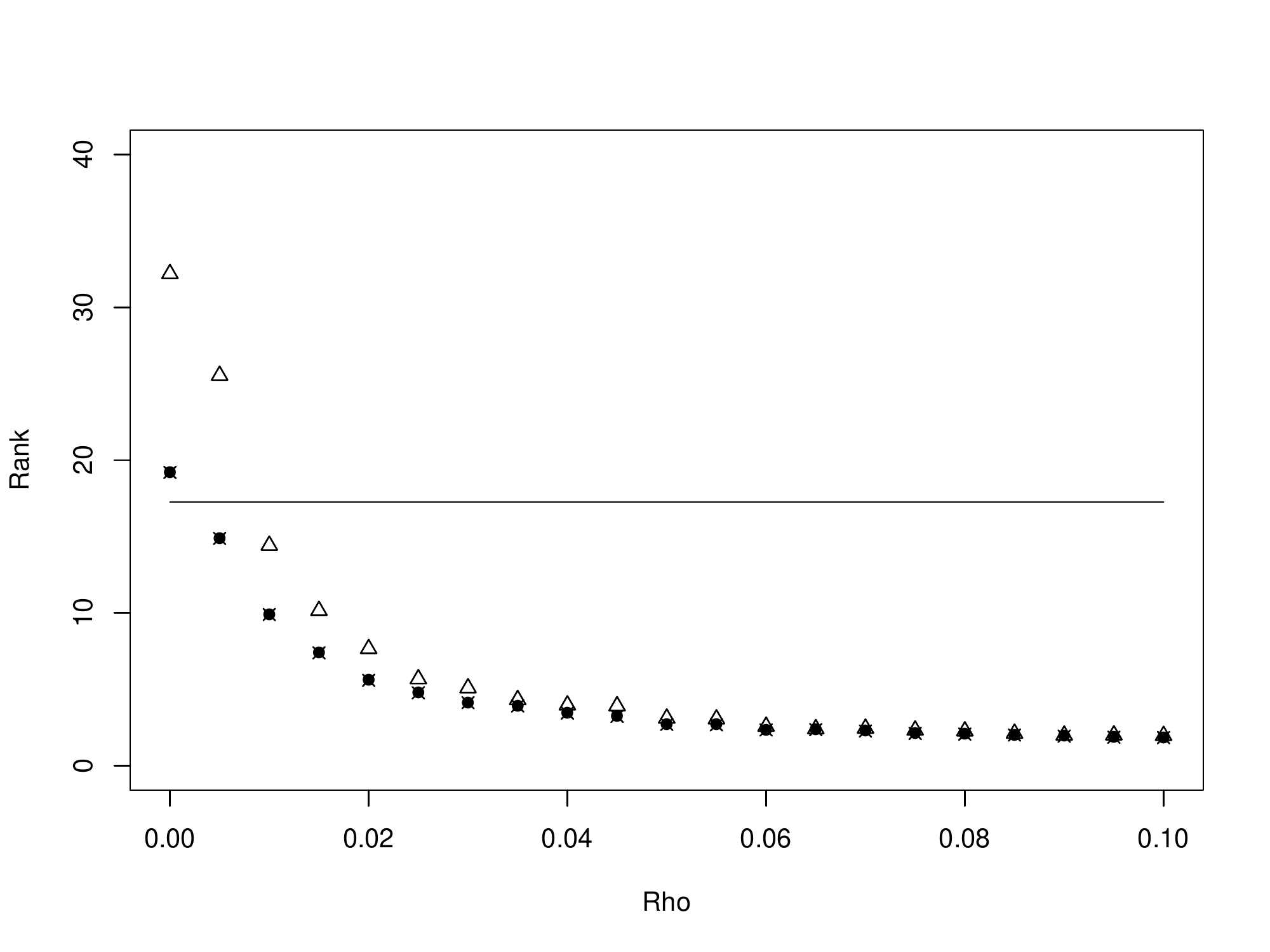}    
\includegraphics[width=0.4\textwidth, height=2.2in]{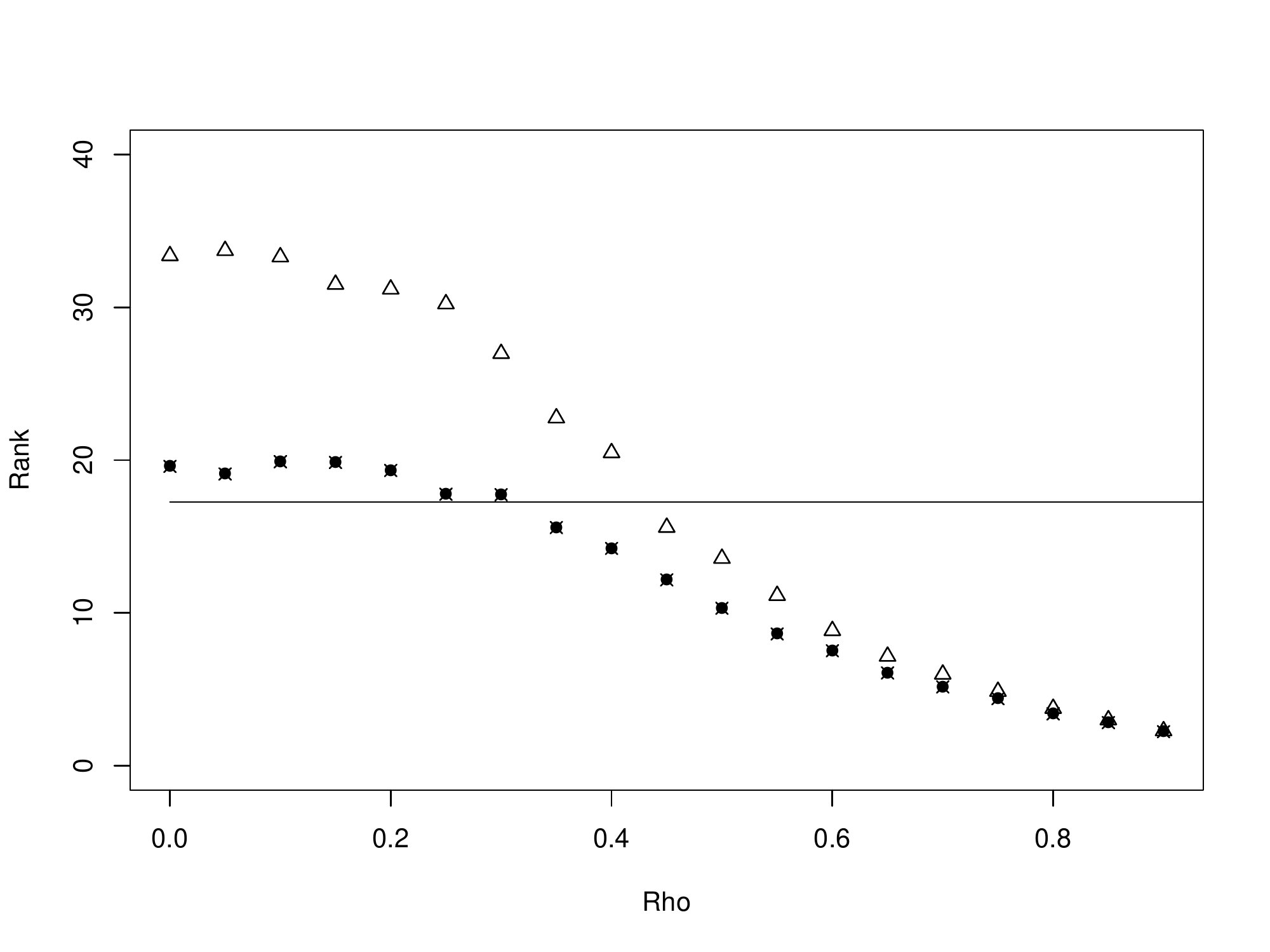}
\caption{Rank of the first spurious variable in three studies. Averaged over 500 replicates, the ranks of forward stepwise, the lasso, and least angle regression are marked with triangles, dots, and crosses, respectively (the dots and the crosses overlap exactly so they look like solid squares). The solid lines indicate the predictions given by Theorem \ref{thm:lower}. Note that for (c,d,e,f) the prediction is constant since $k = 80$ is fixed.}
\label{fig:study}
\end{figure}

\subsection{HIV data}
\label{sec:an-example-from}
As a real data example, we consider the HIV-1 data introduced by \citet{rhee2006genotypic} to study the genetic basis of HIV-1 resistance to several drugs. Also used in a number of other works \citep{barber2015controlling,g2016sequential,janson2016familywise}, this data set in particular contains genotype information $\bX \in \R^{634\times 463}$ of $634$ HIV-1-infected individuals across $463$ locations after removing duplicate and missing values. The columns of $\bX$ are standardized to have zero mean and unit Euclidean norm. The response $\by$ is synthetically generated by assigning an effect of $100\sqrt{2\log p}$ to each of $k$ uniformly randomly chosen columns of $\bX$ and setting a noise level $\sigma$ to 1.

Table \ref{tab:hiv} reports the results averaged over 500 replicates. The three methods start to have a decreasing rank around $k = 25$, which is much smaller than $n/(2\log p) = 51.6$. In addition, for each level of sparsity, the first spurious variable is included much earlier than the predictions. This gap is not surprising given that the predictions are tailored to independent Gaussian designs while the design $\bX$ from the HIV-1 data has strongly correlated columns. To be more precise, about $4600$ column pairs of $\bX$ have correlations greater than $10\%$.

\begin{table}[!htbp]
\fontsize{9}{9}\selectfont
\caption{Rank of the first selected noise predictor averaged over 500 runs, with standard errors given in parentheses. The predictions for sparsity no larger than $n/(2\log p) = 51.6$ are just given as $k+1$ and otherwise are given by Theorems \ref{thm:upper} and \ref{thm:lower}. The last row presents the predictions.}
\begin{tabular}{lccccccc}
& 10 & 25 & 40 & 55 & 70 & 85 & 100\\[3pt]
Lasso & 10.4 (2.1) & 15.4 (9.6) & 10.3 (9.2) & 6.8 (6.0) & 5.5 (4.6) & 4.4 (3.8) & 3.7 (3.5)\\
Least angle regression & 10.4 (2.1) & 15.4 (9.6) & 10.3 (9.2) & 6.8 (6.0) & 5.5 (4.6) & 4.4 (3.8) & 3.7 (3.5)\\
Forward stepwise & 10.6 (1.9) & 18.8 (10.0) & 18.5 (15.9) & 13.3 (16.0) & 10.0 (11.7) & 7.5 (8.2) & 7.0 (7.6)\\
Gaussian designs  &  11.0 &  26.0 &  41.0 & 51.3 & 45.7 & 38.3 & 31.8\\
\end{tabular}
\label{tab:hiv}
\end{table}


\section{Visualizing Early Noise Predictors}
\label{sec:improving-via-double}

As discussed in \S~\ref{sec:heur-expl}, the three sequential procedures are marginal correlation-based at the beginning of their solution paths, picking variables essentially according to the correlations with the residuals. In light of this viewpoint, an unbiased or low-bias estimator of the signals $\bbeta$ might provide sequential methods with complementary information for variable selection. The least-squares estimator, if available, is a natural candidate.

We introduce the \textit{double-ranking diagram} to bring together the strengths of sequential methods and low-bias estimators such as the least-squares estimator $\widehat{\bbeta}^{\textnormal{LS}}$. Figure \ref{fig:ranking} presents two instances of this diagram: one is in the same setting as Fig.~\ref{fig:intro} except for a different size $200 \times 180$ and a fixed sparsity $k = 50$, and another is in the same setting as Table \ref{tab:hiv} with a fixed sparsity $k = 60$. For each variable, the horizontal axis represents its rank by a sequential method, and the vertical axis represents its rank by a low-bias estimator. For example, the horizontal rank of the $j$th variable is given according to the magnitude of $|\widehat{\beta}^{\textnormal{LS}}_j|/\sqrt{[(\bX^\top \bX)^{-1}]_{jj}}$: the larger this statistic is, the smaller the rank is. Equivalently, the variables can be ranked using the $t$-values.

\begin{figure}[!htp]
\centering
\includegraphics[scale=0.55]{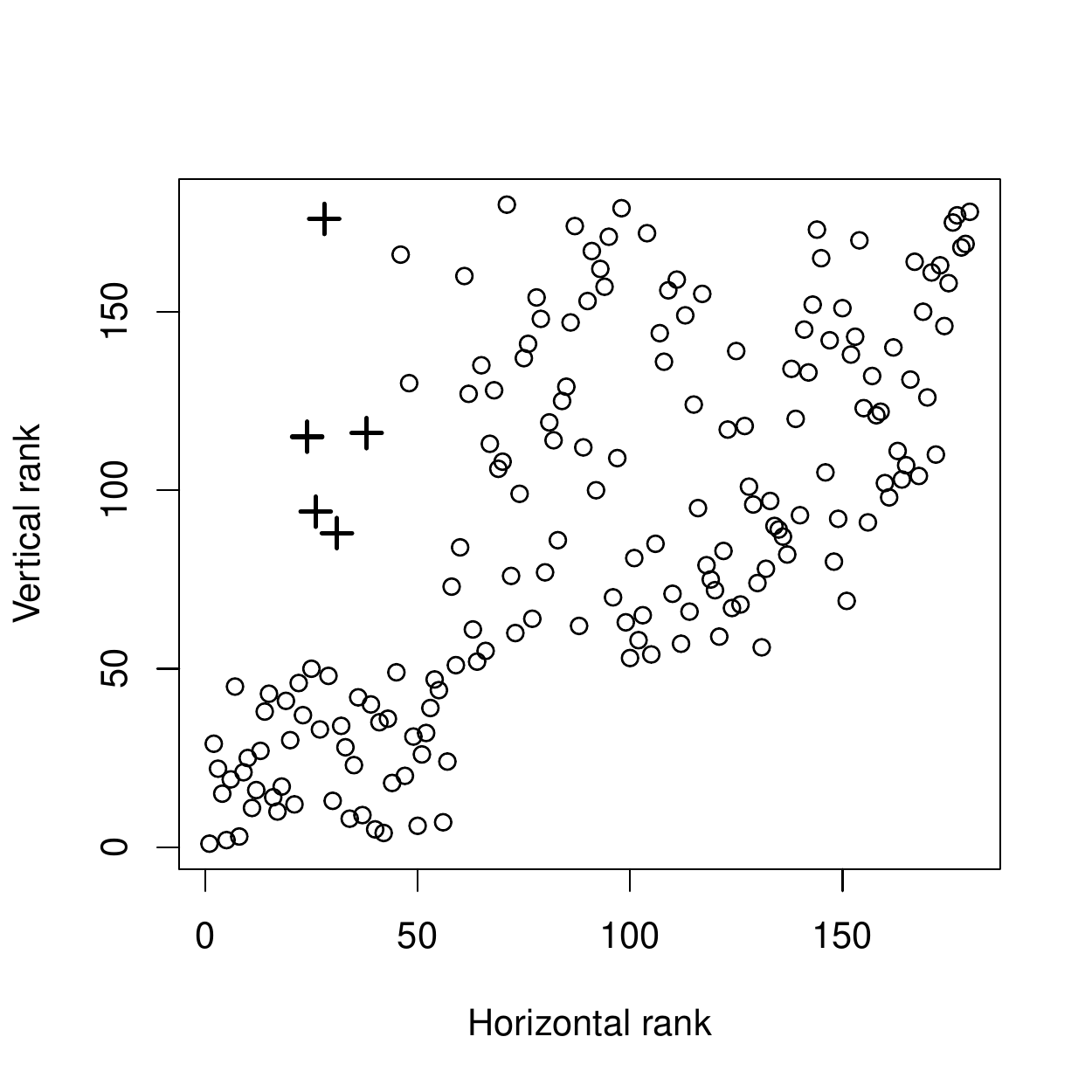}
\includegraphics[scale=0.55]{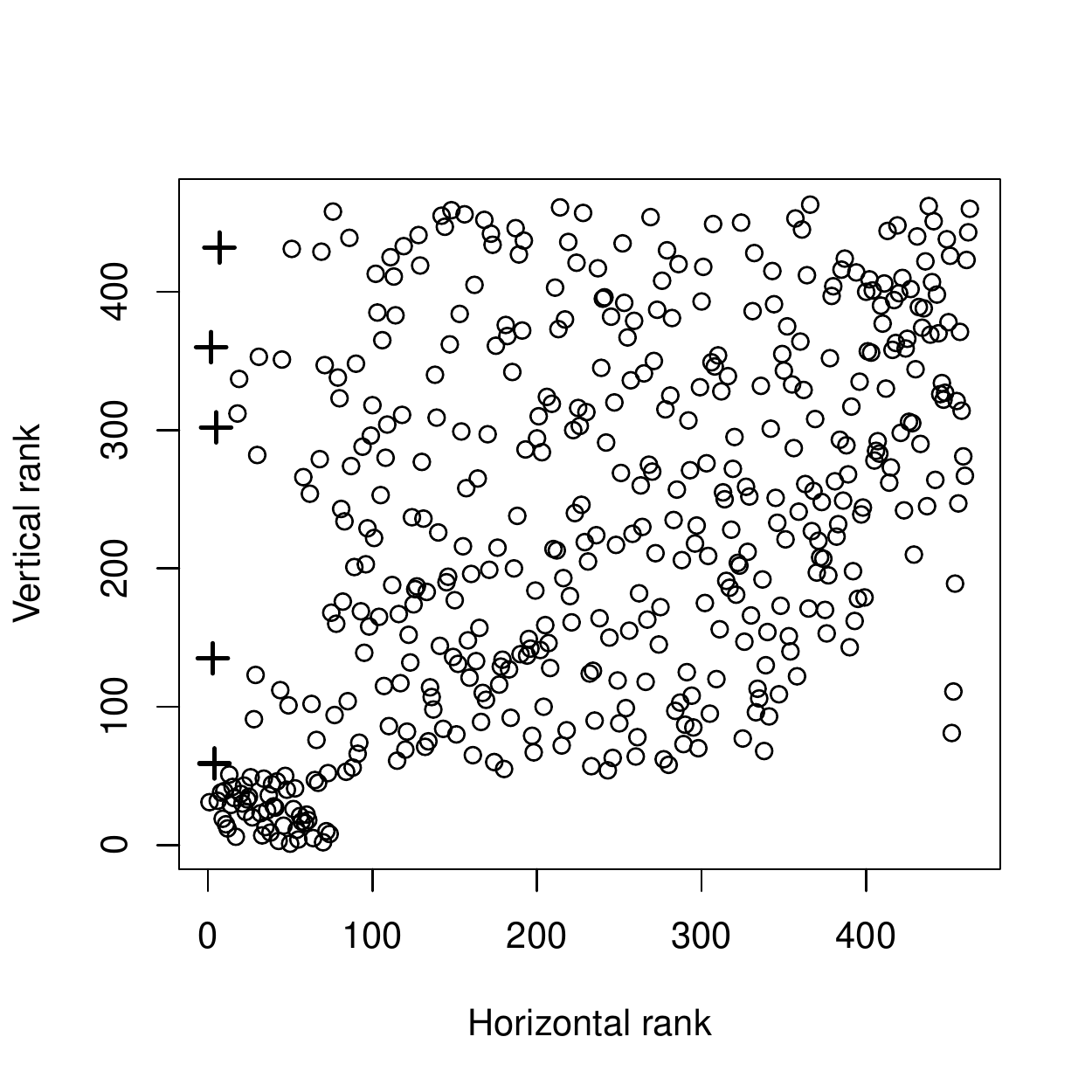}
\caption{Double-ranking diagrams. Left: in the same setting as Fig.~\ref{fig:intro}. Right: in the same setting as the HIV data example. Vertical rankings are given by the least-squares estimators and horizontal rankings are given by the least angle regression. The first five noise variables along the solution path of least angle regression are marked with crosses.}
\label{fig:ranking}
\end{figure}

The double-ranking diagram can serve as a simple data visualization tool to assist the identification of early false variables for sequential regression methods. Intuitively, an important variable would presumably possess both a small horizontal rank and a small vertical rank, hence appearing in the bottom-left corner of the diagram with a good chance. In light of this intuition, we screen out variables that are selected early by a sequential method but have unusually large vertical ranks, which in the case of least squares amount to small $t$-values or insignificant $p$-values. As seen from Fig.~\ref{fig:ranking}, the first five false variables in each instance have much larger vertical ranks compared with their horizontal ranks. In particular, these false variables are placed far away from the signal variables in the diagram. In view of this example, to use this diagram, one can set some threshold for the vertical rank and only select variables that are below the threshold and in addition have significant horizontal ranks. On the other hand, in the low signal-to-noise ratio regime the diagram may not give a clear-cut separation between false and true predictors, and its use requires some caution. The following simple proposition states that the diagram can perfectly separate the first spurious variable from all the true variables using the least-squares estimator under certain conditions.

\begin{proposition}\label{prop:diagram}
Under Assumption \ref{ass:working} and provided that $n > \delta p$ and 
\[
\frac{M}{\sigma} > 3\sqrt{\frac{2\delta\log p}{\delta - 1}}
\] 
for some constant $\delta > 1$, then in the double-ranking diagram the first noise variable has a greater vertical rank than all the true variables.
\end{proposition}

The main ingredient behind this tool is a blend of new and old ideas found in the literature. On the one hand, our discussion in \S~\ref{sec:heur-expl} demonstrates that, while sequential methods work well in very sparse settings, as the signals get denser, the pseudo noise can accumulate quickly and thus may dwarf some true signals no matter how strong the corresponding coefficients are. On the other hand, the method of least squares favors the case of dense signals since the estimator variances basically stay the same as the sparsity of the signals increases. In particular, variables with sufficiently strong effects can stand out using the least-squares estimator in the presence of highly correlated columns in the design matrix. This property of the least-squares estimator and its
variants plays a pivotal role for a number of variable screening procedures \citep{wasserman2009high,pokarowski2015combined,wang2016high}.


\section{Discussion}
\label{sec:discussion}

In the regime of non-extreme sparsity, the common intuition that sequential regression procedures find a significant portion of all important variables before the first false variable merits some skepticism. We have developed sharp predictions that disprove this intuition for forward stepwise, the lasso, and least angle regression under independent Gaussian designs, which satisfy certain desirable properties for model selection. Additionally, the predictions hold irrespective of how strong the effect sizes are. Thus, the first noise variable is likely to occur very early in more general settings. Our numerical results are in agreement with this viewpoint.

In light of the above, more caution is required when using these sequential methods, unless the true regression coefficients are very sparse. Useful information for identifying early noise variables can be provided by low-bias methods such as the least-squares estimators. The double-ranking diagram is a simple tool that unifies the strengths of the two groups of methods.

Avenues for further investigation are in order. First, it is of interest to improve the predictions for forward stepwise and extend the predictions to more sequential methods such as backward stepwise and forward-backward stepwise. The simulation studies imply that the lower bound $c_4 n$ on the sparsity $k$ in Assumption \ref{ass:working} could be possibly relaxed to $n/(2\log p)$. Second, in the high-dimensional setting where $p > n$, which low-bias estimator should we choose for the double-ranking diagram to yield the vertical ranking? Candidates worth considering include the lasso with a small penalty, ridge regression with a small penalty, generalized least-squares estimators (see, for example, \citet{wang2016high}), and some recently proposed ranking procedures such as in
\cite{ke2017covariate}. It is also worth incorporating strategies proposed by \citet{fan2015discoveries} and \citet{fan2016guarding} to investigate spurious discoveries. Last, as seen from Table \ref{tab:hiv}, the rank of the first noise variable has relatively large standard errors. A question of practical relevance is to characterize this large variation.


\section*{Acknowledgements}
The author is grateful to Jianqing Fan, the editor, associate editor and two referees for their comments that improved the presentation of the paper. This work was supported in part by the National Science Foundation via grant CCF-1763314.

\bibliographystyle{plainnat}
\bibliography{ref}

\appendix

\clearpage
\section{Proofs}
\label{sec:appendix}

The appendix is devoted to proving the main technical results in the paper, namely Theorem \ref{thm:upper}, Theorem \ref{thm:lower}, Theorem \ref{thm:unique}, and Proposition \ref{prop:diagram}. Here we collect some notation used in the proofs. Denote by $S$ the true support set, that is, $S = \{j: \beta_j \ne 0\}$. Let $\bX_S$ be the matrix formed by columns from $S$, and $\bX_{-j}$ be the matrix derived by removing the $j$th column from $\bX$. We often use the letter $\bm\mu$ to denote $\bX \bbeta \equiv \bX_S \bbeta_S$, the signal part in linear regression. Throughout the Appendix, assume $M > 0$ in Assumption \ref{ass:working} and adopt the following notation:
\begin{equation}\label{eq:gamma_def}
\Gamma =  \frac{\bbeta^\top \bX^\top\by}{\sqrt{k}M\|\by\|} \equiv  \frac{\|\bm\mu\|^2 + \bz^\top \bm\mu}{\sqrt{k}M\|\by\|}, \quad D \equiv \frac{\|\by\|}{\sqrt{k}},
\end{equation}
where $\|\cdot\|$ denotes the usual $\ell_2$ norm $\|\cdot\|_2$.

\subsection{Theorem \ref{thm:upper}}
\label{sec:proof-theor-refthm:l-1}

We state some preparatory lemmas for the proof of Theorem \ref{thm:upper}. The proofs of these lemmas are given once the proof of Theorem \ref{thm:upper} for all the three sequential procedures is complete.

\begin{lemma}\label{lm:gamma}
Under Assumption \ref{ass:working}, for an arbitrary constant $c > 0$, we have
\[
W-c \le \Gamma \le 1 + c
\]
with probability tending to one.
\end{lemma}

\begin{lemma}\label{lm:normal_log}
Let $\zeta_1, \ldots, \zeta_m$ be independent standard normals and $\zeta_{(1)} \ge \cdots \ge \zeta_{(m)}$ be the order statistics. For any (deterministic) sequence $\{i_m\}$ such that $i_m/m \goto 0$ as $m \goto \infty$, we have
\[
\zeta_{(i_m)} = \sqrt{2\log \frac{m}{i_m}} - (1 + o_{\P}(1))\frac{\log\log\frac{m}{i_m}}{2\sqrt{2\log\frac{m}{i_m}}}.
\]
\end{lemma}
The proof of Lemma \ref{lm:normal_log} is omitted. Interested readers can find its proof in Chapter 2 of \cite{extremevalue}.

\begin{lemma}\label{lm:max_corr}
Under Assumption \ref{ass:working}, we have
\[
\max_j |\bX_j ^\top \by| \le 2D\sqrt{\frac{2k \log p}{n}}
\]
with probability converging to one.
\end{lemma}

\begin{lemma}[for the lasso case]\label{lm:small_proj}
Fit $\by$ on the true support $\bX_S$ using the lasso and denote by $\widehat\bbeta^S(\lambda)$ the lasso solution with penalty $\lambda$. Then, under Assumption \ref{ass:working}, there exists a constant $C$ such that
\[
\max_{i:\widehat\bbeta_i^S(\lambda) = 0} \left| \bX_i^\top \bX \widehat\bbeta^S(\lambda)\right| \le 
\frac{C\sqrt{k} \|\widehat\bbeta^S(\lambda)\|_0\log p}{n} \cdot D
\]
holds uniformly for all $\lambda > 0$ with probability tending to one, where we make the (unusual) convention that $\max \emptyset = 0$.
\end{lemma}

Above, $\|\cdot\|_0$ equals the number of nonzero components of a vector, and $\widehat\bbeta^S$ is a $p$-dimensional vector which takes $0$ on $\overline S := \{1, \ldots, p\}\setminus S$. Next, we proceed to a definition that is the subject of Lemma \ref{lm:IJbound}.

\begin{definition}\label{def:only}
Let $I_c(\gamma)$ be the largest positive integer $I$ such that
\[
\gamma + \sqrt{\frac{2k\log(k/I)}{n}} > \sqrt{\frac{2k \log(p-k)}{n}} - c,
\]
and set $I_c(\gamma)$ to $0$ if it does not exist. Let $J_c(\gamma)$ be the largest positive integer $J$ such that
\[
-\gamma + \sqrt{\frac{2k\log(k/J)}{n}} > \sqrt{\frac{2k \log(p-k)}{n}} - c,
\]
and set $J_c(\gamma)$ to $0$ if it does not exist. 
\end{definition}

\begin{lemma}\label{lm:IJbound}
Let $c$ be a constant. Under Assumption \ref{ass:working}, all the following statements are true.
\begin{itemize}
\item[(a)] Assume $c > 0$. If $-c/2 < \gamma \le 0$, then $J_c(\gamma) \ge I_c(\gamma)$, and
\[
\log J_c(\gamma) = (1+o(1))\left[ (c-\gamma)\sqrt{\frac{2n\log p}{k}} - \frac{(c-\gamma)^2n}{2k} + \log\frac{n}{2p \log p}\right].
\]

\item[(b)] Assume $c > 0$. If $0 < \gamma \le 1.01$ and $\sqrt{2k \log(p-k)/n} >  c + \gamma + c'$
for some constant $c' > 0$ (which is used to guarantee that $I_c(\gamma) \ne k$), then we have $I_c(\gamma) \ge J_c(\gamma)$ and
\[
\log I_c(\gamma) = (1+o(1))\left[(c+\gamma)\sqrt{\frac{2n\log p}{k}} - \frac{(c+\gamma)^2n}{2k} + \log\frac{n}{2p \log p}\right].
\]

\item[(c)]
Assume $-0.1 < c < 0$. If $0.11 < \gamma \le 1.01$, then $I_c(\gamma) \ge J_c(\gamma)$ and
\[
\log I_c(\gamma) = (1+o(1))\left[(c+\gamma)\sqrt{\frac{2n\log p}{k}} - \frac{(c+\gamma)^2n}{2k} + \log\frac{n}{2p \log p}\right].
\]
\end{itemize}
\end{lemma}

A detailed comment on how $I_c(\gamma)$ and $J_c(\gamma)$ are used in proofs is as follows. Under our Assumption \ref{ass:working}, the rank of a variable $\bX_i$ roughly depends on the absolute value of $\bX_i^\top \by$. The random variable $\bX_i^\top \by$, as seen later in the proof of Theorem \ref{thm:upper}, is approximately distributed as $D(\Gamma + \sqrt{k/n} \N(0, 1))$ (note that $\Gamma$ is defined in \eqref{eq:gamma_def}). Intuitively, taking $\gamma = \Gamma$, the first display of Definition \ref{def:only} represents the event that the $I$th true variables along the solution path has an inner product with $\by$ about equal to that of the first false variable (recognize from Lemma \ref{lm:normal_log} that the $I$th largest of $k$ independent $\N(0, 1)$ random variables is about
$\sqrt{2\log(k/I)}$ and the largest of $p-k$ independent $\N(0, 1)$ random variables is about $\sqrt{2\log(p-k)}$). The notation $J_c(\gamma)$ is introduced because $D(\Gamma + \sqrt{k/n} \N(0, 1))$ can also take a large magnitude if $\N(0, 1)$ in the parentheses is about $-\sqrt{2\log(k/J)}$.

Henceforth, $\gamma$ is set to $\Gamma$ and the dependence on the argument is omitted for both $I_c$ and $J_c$. Due to Lemma \ref{lm:gamma}, it is without loss of generality to consider $\Gamma \in (-o(1), 1+ o(1))$ and $c$ fixed and small.


With these preparatory lemmas in place, we are ready to prove Theorem \ref{thm:upper} in the lasso and least angle regression cases. The first part of the proof offers a simple representation of the order statistics of $\bX_S^\top \by$, and this representation serves as useful ingredients in proofs of other results.
\begin{proof}[Proof of Theorem \ref{thm:upper} in the lasso and least angle regression cases]
Without loss of generality, assume that the true support set $S = \{1, \ldots, k\}$. Let $(1), (2), \ldots, (k)$ be a permutation of $1, \ldots, k$ such that
\begin{equation}\label{eq:order}
\bX_{(1)}^\top \by  \ge \bX_{(2)}^\top \by \ge \cdots \ge \bX_{(k)}^\top \by.
\end{equation}
Conditional on $\bz$ and $\bX_S \bbeta_S \equiv \bm\mu$, the $k$ exchangeable random variables $\bX_1^\top\by, \ldots, \bX_k^\top \by$ are jointly normal with means all equal to $\bm\mu^\top \by/(kM)$ and an equicorrelated covariance that has $(k-1)\|\by\|^2/(kn)$ on the diagonal and $-\|\by\|^2/(kn)$ off the diagonal. These can be derived by using properties of the conditional normal distribution: suppose $W_1, \ldots, W_m$ are iid standard normals, then $W_1$ is normally distributed with mean $C/m$ and variance $(m-1)/m$ conditional on $W_1 + \cdots + W_m = C$. Let $\xi$ be normally distributed with mean 0 and variance $\|\by\|^2/(kn)$, and further assume $\xi$ to be independent of the random design $\bX$ and the noise $\bz$. Then, $\bX_1^\top \by  + \xi, \ldots, \bX_k^\top \by + \xi$ are independent normals each with conditional mean $\bm\mu^\top \by/(kM)$ and conditional variance $\|\by\|^2/n$. Thus, conditional on $\bm\mu$ and $\bz$, we see that
\[
\begin{aligned}
\bX_i^\top \by  + \xi &\overset{d}{=} \frac{\bm\mu^\top \by}{kM} + \frac{\|\by\|}{\sqrt{n}} \zeta_i \\
&= D\left( \Gamma +  \sqrt{\frac{k}{n}} \zeta_i \right)
\end{aligned}
\]
for independent standard normals $\zeta_1, \ldots, \zeta_k$. Note that $\bX_1^\top \by  + \xi, \ldots, \bX_k^\top \by + \xi$ keep the same ordering as in \eqref{eq:order}.

Consider the first time that the full lasso (regressing on $\bX$) is just about to include some variable among $\bX_{(I_c+1)}, \bX_{(I_c+2)}, \ldots, \bX_{(k - J_c)}$, where $I_c = I_c(\Gamma)$ and $J_c = J_c(\Gamma)$ as in Definition \ref{def:only}. Call this variable $\bX_{(L)}$ and the penalty at that time $\lambda^\ast$. Denote by $\A$ the event that all the selected variables preceding $\bX_{(L)}$ are true variables. On $\overline\A$ (the complement of $\A$), there are at most $I_c + J_c$ variables selected before the first false variable. Hence, we get
\[
T \le I_c + J_c + 1
\]
on $\overline{\A}$, where $c > 0$ is an arbitrary constant. By Lemma \ref{lm:IJbound}, we get
\[
\begin{aligned}
\log(I_c + J_c + 1) &= \log\max\{I_c, J_c\} + O(1)\\
&=(1+o_{\P}(1))\left[ (c+|\Gamma|)\sqrt{\frac{2n\log p }{k}} - \frac{(c+|\Gamma|)^2n}{2k} + \log\frac{n}{2p \log p}\right] + O(1)\\
&=(1+o_{\P}(1))\left[(c+|\Gamma|)\sqrt{\frac{2n\log p }{k}} - \frac{(c+|\Gamma|)^2n}{2k} + \log\frac{n}{2p \log p} \right]\\
& \le (1+o_{\P}(1))\left[(1 + 3c/2)\sqrt{\frac{2n\log p }{k}} - \frac{(1+3c/2)^2n}{2k} + \log\frac{n}{2p \log p}\right]\\
\end{aligned}
\]
with probability approaching one. Above, the $O(1)$ term is absorbed into $\log\max\{I_c, J_c\}$. Lemma A\ref{lm:gamma} is also used in the inequality following the third equality, with $c/2$ in place of $c$. Since $c > 0$ is arbitrary, we get
\[
\log T \le \log(I_c + J_c + 1)  \le (1+o_{\P}(1))\left[ \sqrt{\frac{2n\log p }{k}} - \frac{n}{2k} + \log\frac{n}{2p \log p}\right].
\]
on $\overline{\A}$.

The rest is devoted to proving $\P(\A) \goto 0$. To proceed, we point out an observation on the lasso: the lasso running on the full design $\bX$ is the same as regressing on $\bX_S$ until the $T$th variable gets selected. In light of this observation, we perform the lasso on the true support $\bX_S$. Since $I_c + J_c$ is much smaller than the bound given by Theorem \ref{thm:unique}, we can assume  no drop-out has happened until $\bX_{(L)}$ arrives and, hence, least angle regression is the same as the lasso. Because $I_c +1 \le L \le k - J_c$, we get
\begin{equation}\label{eq:L_sandwich}
\bX_{(I_c+1)}^\top \by \ge \bX_{(L)}^\top \by \ge \bX_{(k - J_c)}^\top \by.
\end{equation}
By Lemma \ref{lm:normal_log}, the left-hand side of \eqref{eq:L_sandwich} obeys 
\[
\bX_{(I_c+1)}^\top \by  = D \left[\Gamma + \sqrt{\frac{2k\log(k/(I_c+1))}{n}} - (1+o_{\P}(1))\frac{\sqrt{k}\log\log(k/(I_c+1))}{2\sqrt{2n\log(k/(I_c+1))}} - \frac{\xi}{D}\right]
\]
A little analysis reveals that Assumption \ref{ass:working} implies
\[
\frac{\sqrt{k}\log\log(k/(I_c+1))}{2\sqrt{2n\log(k/(I_c+1))}} = o_{\P}(1).
\]
Hence, we get
\[
\begin{aligned}
\bX_{(I_c+1)}^\top \by  &= D \left[\Gamma + \sqrt{\frac{2k\log(k/(I_c+1))}{n}} - o_{\P}(1) - \frac{\xi}{D}\right]  \\
&= D \left[\Gamma + \sqrt{\frac{2k\log(k/(I_c+1))}{n}} + o_{\P}(1)\right]\\
&\le D \left[\sqrt{\frac{2k\log(p-k)}{n}} - c  + o_{\P}(1)\right].
\end{aligned}
\]
Similarly, the right-hand side of \eqref{eq:L_sandwich} obeys
\[
\begin{aligned}
\bX_{(k-J_c)}^\top \by &= D \left[\Gamma - \sqrt{\frac{2k\log(k/(J_c+1))}{n}} + o_{\P}(1) \right]\\
&\ge D \left[-\sqrt{\frac{2k\log(p-k)}{n}} + c + o_{\P}(1)\right]\\
&= -D \left[\sqrt{\frac{2k\log(p-k)}{n}} - c + o_{\P}(1)\right].
\end{aligned}
\]
Thus, from \eqref{eq:L_sandwich} it follows that
\[
D \left[\sqrt{\frac{2k\log(p-k)}{n}} - c + o_{\P}(1)\right] \ge \bX_{(L)}^\top \by \ge -D \left[\sqrt{\frac{2k\log(p-k)}{n}} - c + o_{\P}(1)\right],
\]
yielding
\begin{equation}\label{eq:xbeta_max}
\left|\bX_{(L)}^\top \by \right| \le D \left[\sqrt{\frac{2k\log(p-k)}{n}} - c + o_{\P}(1)\right].
\end{equation}
The KKT conditions of the lasso at $\lambda^\ast$ give
\[
\left|\bX_{(L)}^\top \left( \by - \bX \widehat\bbeta^S \right) \right|= \lambda^\ast.
\]
The equality above together with \eqref{eq:xbeta_max} and Lemma \ref{lm:small_proj}, which gives $\bX_{(L)}^\top\bX \widehat\bbeta^S = o_{\P}(D)$, yields
\begin{equation}\label{eq:L_boundA}
\begin{aligned}
\lambda^\ast = \left| \bX_{(L)}^\top \left( \by - \bX \widehat\bbeta^S \right)\right| &\le D \left[\sqrt{\frac{2k\log(p-k)}{n}} - c + o_{\P}(1) + |\bX_{(L)}^\top \bX \widehat\bbeta^S|/D\right]\\
&= D \left[\sqrt{\frac{2k\log(p-k)}{n}} - c + o_{\P}(1)\right].
\end{aligned}
\end{equation}

Now, we seek a contradiction to get $\P(\A) \goto 0$. Since none of $\bX_{k+1}, \ldots, \bX_{p}$ has been included at $\lambda^\ast$ on the event $\A$, we get
\begin{equation}\label{eq:kkt_none}
\max_{k+1 \le j \le p} \left| \bX_{j}^\top \left( \by - \bX \widehat\bbeta^S \right)\right| \le \lambda^\ast \le D \left[\sqrt{\frac{2k\log(p-k)}{n}} - c + o_{\P}(1)\right].
\end{equation}
Recognizing the independence between all $\bX_{j}, j \ge k+1$ and $\by - \bX \widehat\bbeta^S$, we get
\[
\begin{aligned}
\max_{k+1 \le j \le p} \left| \bX_{j}^\top \left( \by - \bX \widehat\bbeta^S \right)\right| &\ge 
\max_{k+1 \le j \le p} \left| \bX_{j}^\top \by\right|  - \max_{k+1 \le j \le p} \left| \bX_{j}^\top \bX \widehat\bbeta^S\right| \\
&= \max_{k+1 \le j \le p} \left| \bX_{j}^\top \by\right|  - o_{\P}(D).
\end{aligned}
\]
Now we turn to focus on $\max_{k+1 \le j \le p} \sqrt{n} | \bX_{j}^\top \by | /\|\by\|$, which is distributed as the maximum absolute values of $p-k$ iid standard normals. Hence, Lemma \ref{lm:normal_log} gives
\[
\max_{k+1 \le j \le p} \frac{\sqrt{n} | \bX_{j}^\top \by |} {\|\by\|} = \sqrt{2\log (p-k)} - \frac{(0.5+o_{\P}(1))\log\log (p-k)}{\sqrt{2\log (p-k)}}.
\]
Consequently,
\begin{equation}\nonumber
\begin{aligned}
\max_{k+1 \le j \le p} \left| \bX_{j}^\top \left( \by - \bX \widehat\bbeta^S \right)\right| &\ge  \max_{k+1 \le j \le p} \left| \bX_{j}^\top \by\right|  - o_{\P}(D)\\
& = D\sqrt{\frac{k}{n}}\left[\sqrt{2\log (p-k)} - \frac{(0.5+o_{\P}(1))\log\log (p-k)}{\sqrt{2\log (p-k)}}\right] - o_{\P}(D)\\
&= D \sqrt{\frac{2k\log(p-k)}{n}} - D\sqrt{\frac{k}{n}}\frac{(0.5+o_{\P}(1))\log\log (p-k)}{\sqrt{2\log (p-k)}} - o_{\P}(D)
\end{aligned}
\end{equation}
By assumption,
\[
\sqrt{\frac{k}{n}}\frac{\log\log (p-k)}{\sqrt{2\log (p-k)}} \lesssim \sqrt{\log^{0.99} p} \cdot \frac{\log\log (p-k)}{\sqrt{2\log (p-k)}} \asymp \sqrt{\log^{0.99} p} \cdot \frac{\log\log p}{\sqrt{2\log p}} = o(1).
\]
Therefore, we get
\begin{equation}\nonumber
\max_{k+1 \le j \le p} \left| \bX_{j}^\top \left( \by - \bX \widehat\bbeta \right)\right| \ge D \left[\sqrt{\frac{2k\log(p-k)}{n}} + o_{\P}(1) \right],
\end{equation}
contradicting \eqref{eq:kkt_none}. Thus, $\pr(\A) = o(1)$.

\end{proof}

Next, we aim to prove Theorem \ref{thm:upper} in the case of forward stepwise. Below are two preparatory lemmas.
\begin{lemma}\label{lm:small_proj_stepwise}
Given $1 \le m \le p$, for any subset $F \subset \{1, \ldots, p\}$ of cardinality at most $m$, $i \notin F$, and $j \in F$, regress $\bX_i$ on $\bX_F$ and denote by $\widehat\alpha_{j}^{i,F}$ the least-squares coefficient of $\bX_j$ in this fit. Then, under Assumption \ref{ass:working}, with probability approaching one, we have
\[
|\widehat\alpha_j^{i,F}| \le \sqrt{\frac{5m\log p}{n}}
\]
uniformly over all $i, j$, and $F$ of cardinality at most $m$.
\end{lemma}

\begin{lemma}\label{lm:chi_sub_expo}
Under Assumption \ref{ass:working}, we get
\[
1 - 3\sqrt{\frac{\log p}{n}} \le \|\bX_i\| \le 1 + 3\sqrt{\frac{\log p}{n}}.
\]
with probability approaching one uniformly for all $1 \le i \le p$.
\end{lemma}
This lemma is a simple consequence of a well-known concentration inequality of the form
\[
\P\left(\left|\frac{\chi^2_n}{n} - 1 \right| \ge t \right) \le 2 \e^{-n t^2/8}
\]
for all $0 < t < 1$. The proof is thus omitted. Now we move to prove Theorem \ref{thm:upper} for forward stepwise regression.

\begin{proof}[Proof of Theorem \ref{thm:upper} in the forward stepwise case]
Here, we use the same proof strategy as in the lasso and least angle regression cases. Define $\A$ and $\bX_{(L)}$ as earlier. To complete the proof, it suffices to show that $\P(\A) \goto 0$.

We also assume that forward stepwise is performed on the true support $\bX_S$. Denote by $m^\ast$ the number of variables selected before $\bX_{(L)}$, that is, $\bX_{(L)}$ is selected in the $(m^\ast+1)$th step. Take some $m$ satisfying $m \le p^c$ for some $c < 1/3$. It is easy to show that $m^\ast \ll m$ with probability approaching one.

Denote by $\fsset^\ast$ the set of selected variables before $\bX_{(L)}$. On the event $\A$, the definition of forward stepwise yields
\[
\min_{\supp(\bb) = \fsset^\ast \cup (L)}\| \by - \bX \bb\|^2 \le \min_{\supp(\bb) = \fsset^\ast \cup l, l \in \overline{S}}\| \by - \bX \bb\|^2.
\]
Recall that $\overline{S} = \{k+1, \ldots, p\}$ stands for the complement of $S$. In particular, we have
\begin{equation}\label{eq:more_descent}
\min_{\supp(\bb) = \fsset^\ast \cup (L)}\| \by - \bX \bb\|^2 \le \min_{\supp(\bb) = \fsset^\ast \cup j'}\| \by - \bX \bb\|^2,
\end{equation}
where $j' = \argmax_{j \in \overline{S}} |\bX_j^\top \by|$. Let $\widetilde\bX_{(L)}$ be residual vector by regressing $\bX_{(L)}$ on $\bX_{F^\ast}$. Then,
\[
\begin{aligned}
&\min_{\supp(\bb) = \fsset^\ast \cup (L)}\| \by - \bX \bb\|^2 - \min_{\supp(\bb) = \fsset^\ast}\| \by - \bX \bb\|^2\\
&= - \left[\frac{\widetilde\bX_{(L)}^\top \by}{\|\widetilde\bX_{(L)}\|}\right]^2\\
&= - \left[\frac{(\bX_{(L)}^\top - \sum_{i \in F^\ast}\alpha_i \bX_i^\top)  \by}{\|\widetilde\bX_{(L)}\|}\right]^2,
\end{aligned}
\]
where $\alpha_i$'s are the least-squares coefficients. By Lemma \ref{lm:small_proj_stepwise}, we get
\[
\begin{aligned}
\left\|\sum_{i \in F^\ast}\alpha_i \bX_i^\top \by \right\| &\le \sqrt{m^\ast}\max_{i \in F^\ast} \|\alpha_i X_{i}^\top \by\|\\
&\le \sqrt{m}\max_{i \in F^\ast} \|\alpha_i X_{i}^\top \by\|\\
&\le \sqrt{m} \sqrt{\frac{5m\log p}{n}} \max_{i \in F^\ast} \|X_{i}^\top \by\|\\
&\le \sqrt{\frac{5m^2\log p}{n}} \cdot 2D \sqrt{\frac{2k\log p}{n}}\\
&= D\sqrt{\frac{40 m^2 k\log^2 p}{n^2}}\\
& \lesssim D\sqrt{\frac{40m^2\log^{2.99} p}{n}}\\
& = o(D),
\end{aligned}
\]
which makes use of Lemma \ref{lm:max_corr}. Therefore, we get
\[
\left( \bX_{(L)}^\top - \sum_{i \in F^\ast}\alpha_i \bX_i^\top \right)  \by = \bX_{(L)}^\top \by + o_{\P}(D).
\]
Together with Lemma \ref{lm:chi_sub_expo}, the equality above gives
\begin{equation}\label{eq:gap1}
\begin{aligned}
\left[\frac{(\bX_{(L)}^\top - \sum_{i \in F^\ast}\alpha_i \bX_i^\top)  \by}{\|\widetilde\bX_{(L)}\|}\right]^2 &\le \left[\frac{\bX_{(L)}^\top \by + o_{\P}(D)}{\|\bX_{(L)}\| - \sum_{i \in F^\ast}|\alpha_i| \|\bX_i\|}\right]^2\\
&\le \left[\frac{\bX_{(L)}^\top \by + o_{\P}(D)}{1 - 3\sqrt{\frac{\log p}{n}} - m\sqrt{\frac{5m\log p}{n}} \left(1 + 3\sqrt{\frac{\log p}{n}} \right)}\right]^2\\
&= (1 + O(\sqrt{m^3 n^{-1} \log p}))\left(\bX_{(L)}^\top \by + o_{\P}(D) \right)^2.
\end{aligned}
\end{equation}
Similarly, we have
\begin{equation}\label{eq:gap2}
\left[\frac{(\bX_{(L)}^\top - \sum_{i \in F^\ast}\alpha_i \bX_i^\top)  \by}{\|\widetilde\bX_{(L)}\|}\right]^2 \ge (1 - O(\sqrt{m^3 n^{-1} \log p}))\left(\bX_{(L)}^\top \by + o_{\P}(D) \right)^2.
\end{equation}
For the right-hand side of \eqref{eq:more_descent}, we have
\begin{multline}\label{eq:out_highcorr}
-(1 + O(\sqrt{m^3 n^{-1} \log p}))\left(\bX_{j'}^\top \by + o_{\P}(D) \right)^2 \\
\le \min_{\supp(\bb) = \fsset^\ast \cup j'}\| \by - \bX \bb\|^2 - \min_{\supp(\bb) = \fsset^\ast}\| \by - \bX \bb\|^2 \\
\le -(1 - O(\sqrt{m^3 n^{-1} \log p}))\left(\bX_{j'}^\top \by + o_{\P}(D) \right)^2
\end{multline}
with probability approaching one. Combining \eqref{eq:more_descent}, \eqref{eq:gap1}, \eqref{eq:gap2}, and \eqref{eq:out_highcorr} gives
\begin{equation}\label{eq:fs_rule}
(1 + O(\sqrt{m^3 n^{-1} \log p}))\left(\bX_{(L)}^\top \by + o_{\P}(D) \right)^2 \ge (1 - O(\sqrt{m^3 n^{-1} \log p}))\left(\bX_{j'}^\top \by + o_{\P}(D) \right)^2
\end{equation}
on the event $\A$.

To complete the proof, we shall show that \eqref{eq:fs_rule} holds with probability tending to zero. On the one hand, from the earlier proof for the case of the lasso, we see that
\[
\begin{aligned}
(1 + O(\sqrt{m^3 n^{-1} \log p}))\left(\bX_{(L)}^\top \by + o_{\P}(D) \right)^2  &\le (1 + O(\sqrt{m^3 n^{-1} \log p})) D^2 \left[\sqrt{\frac{2k\log(p-k)}{n}} - c + o_{\P}(1)\right]^2\\
&= D^2 \left[\sqrt{\frac{2k\log(p-k)}{n}} - c + o_{\P}(1)\right]^2,
\end{aligned}
\]
which makes use of the fact that $O(\sqrt{m^3 n^{-1} \log p})) \cdot \sqrt{2k\log(p-k)/n} = o(1)$ under our assumptions. On the other hand,
\[
\begin{aligned}
(1 - O(\sqrt{m^3 n^{-1} \log p}))\left(\bX_{j'}^\top \by + o_{\P}(D) \right)^2 &= (1 - O(\sqrt{m^3 n^{-1} \log p})) D^2 \left[\sqrt{\frac{2k\log(p-k)}{n}}  + o_{\P}(1)\right]^2\\
&= D^2 \left[\sqrt{\frac{2k\log(p-k)}{n}}  + o_{\P}(1)\right]^2
\end{aligned}
\]
Hence, \eqref{eq:fs_rule} yields
\[
D^2 \left[\sqrt{\frac{2k\log(p-k)}{n}} - c + o_{\P}(1)\right]^2 \ge D^2 \left[\sqrt{\frac{2k\log(p-k)}{n}}  + o_{\P}(1)\right]^2,
\]
which is a clear contradiction. Hence, $\P(\A) \goto 0$.

\end{proof}


To conclude this section, we present proofs of Lemma \ref{lm:gamma} through Lemma \ref{lm:small_proj_stepwise}, respectively.

\begin{proof}[Proof of Lemma \ref{lm:gamma}]

Recall the notation $\bm \mu = \bX\bbeta$. We first list some facts that are constantly used in the proof. First, 
\[
\|\bX \bbeta\| = (1 + o_{\P}(1))M\sqrt{k}.
\]
This is because $n\|\bX \bbeta\|^2/(kM^2)$ is just a $\chi^2$ random variable with $n$ degrees of freedom. Thus, $n\|\bX \bbeta\|^2/(kM^2) = (1 + o_{\P}(1))n$, yielding $\|\bX \bbeta\| = (1+o_{\P}(1)) M\sqrt{k}$. Hence, we get
\[
\Gamma = \frac{\bbeta^\top \bX^\top\by}{\sqrt{k}M\|\by\|}  \le \frac{\|\bbeta^\top \bX^\top\| \|\by\|}{\sqrt{k}M\|\by\|} = \frac{\|\bX \bbeta\|}{\sqrt{k}M} \le 1 + o_{\P}(1).
\]

Next we turn to prove that $\Gamma \ge -c$.
In fact, we have
\begin{equation}\label{eq:gamma_larger}
\Gamma \equiv \frac{\|\bm\mu\|^2 + \bm\mu^\top \bz}{\sqrt{k}M\|\by\|} \ge \frac{\bm\mu^\top \bz}{\sqrt{k}M\|\by\|}.
\end{equation}
For $\|\by\|$, note that
\[
\|\by\| = (1 + o_{\P}(1)) \sqrt{kM^2 + n\sigma^2}  \ge (1 + o_{\P}(1)) \sigma\sqrt{n}.
\]
And conditional on $\bm\mu$, the numerator $\bm\mu^\top \bz$ is normally distributed with mean 0 and variance $\sigma^2 \|\bm\mu\|^2$, implying
\[
\bm\mu^\top \bz = O_{\P}(\sigma \|\bm\mu\|) = O_{\P}(M \sigma \sqrt{k}).
\]
Combining the results above, particularly \eqref{eq:gamma_larger}, gives
\[
\Gamma \ge (1 + o_{\P}(1))\frac{O_{\P}(M \sigma \sqrt{k})}{\sqrt{k}M\sigma\sqrt{n}} = O_{\P}(1/\sqrt{n}) = o_{\P}(1).
\]
Hence, $\Gamma \ge -c$ with probability approaching one for an arbitrary constant $c > 0$.
\end{proof}

\begin{proof}[Proof of Lemma \ref{lm:max_corr}]
If $\beta_j = 0$, then $\bX_j$ is independent of $\by$. This means conditional on $\by$ the distribution of $\bX_j ^\top \by$ is $\N(0, \|\by\|^2/n)$. Thus,
\[
|\bX_j ^\top \by| \le \frac{\|\by\|}{\sqrt{n}} \sqrt{2\log p} = D\sqrt{\frac{2k \log p}{n}}
\]
with probability at least $1 - o(1/p)$. Taking a union bound yields
\[
\max_{j \notin S} |\bX_j ^\top \by| \le  D\sqrt{\frac{2k \log p}{n}} < 2D\sqrt{\frac{2k \log p}{n}}
\]
with probability $1 - o(1)$. Now consider a $j$ such that $\beta_j \ne 0$. Note that
\begin{equation}\label{eq:xy_two}
\bX_j ^\top \by = \bX_j ^\top \bX_j \beta_j + \bX_j ^\top (\bX_{-j}\bbeta_{-j} + \bz) = M\bX_j ^\top \bX_j + \bX_j ^\top (\bX_{-j}\bbeta_{-j} + \bz).
\end{equation}
The first term obeys
\[
\P\left( |\bX_j ^\top \bX_j - 1| \ge t \right) = \P\left( |\chi^2_{n}/n- 1| \ge t \right) \le 2 \e^{-n t^2/8}
\]
for any $0 < t < 1$. Setting $t$ to $n^{-\frac14}$ gives
\[
\max_j |\bX_j ^\top \bX_j - 1| = o_{\P}(1).
\] 
In addition, we have
\[
\begin{aligned}
D\sqrt{\frac{2k \log p}{n}} &= (1 + o_{\P}(1)) \sqrt{\frac{k M^2 + n \sigma^2}{k}} \sqrt{\frac{2k \log p}{n}} \\
&\ge (1 + o_{\P}(1)) \sqrt{\frac{k M^2}{k}} \sqrt{\frac{2k \log p}{n}}\\
&\ge (1 + o_{\P}(1)) M \sqrt{\frac{2k \log p}{n}}.\\
\end{aligned}
\]
Recognizing that $\sqrt{2k (\log p)/n} \goto \infty$, we get
\begin{equation}\label{eq:first_small}
\left| M\bX_j ^\top \bX_j \right| = (1 + o_{\P}(1))M = o_{\P}(1) \cdot D\sqrt{\frac{2k \log p}{n}}.
\end{equation}
For the second term of \eqref{eq:xy_two}, note that $\bX_j ^\top (\bX_{-j}\bbeta_{-j} + \bz)$ is distributed as a centered normal random variable with variance $\|\bX_{-j}\bbeta_{-j} + \bz\|^2/n$ conditional on $\bX_{-j}\bbeta_{-j} + \bz$, due to the independence between $\bX_j$ and $\bX_{-j}\bbeta_{-j} + \bz$. Consequently, $|\bX_j ^\top (\bX_{-j}\bbeta_{-j} + \bz)| \le \|\bX_{-j}\bbeta_{-j} + \bz\| n^{-1/2} \sqrt{2\log p}$ holds with probability $1 - o(1/p)$, from which we get
\[
\max_j |\bX_j ^\top (\bX_{-j}\bbeta_{-j} + \bz)| \le (1+o_{\P}(1))\frac{\|\bX_{-j}\bbeta_{-j} + \bz\|}{\sqrt n} \sqrt{2\log p} = (1 + o_{\P}(1))D\sqrt{\frac{2k \log p}{n}},
\]
which makes use of the fact that $\|\bX_{-j}\bbeta_{-j} + \bz\| = (1 + o_{\P}(1))\sqrt{k} D$ uniformly for all $j$. Plugging the last display and \eqref{eq:first_small} into \eqref{eq:xy_two} gives
\[
\max_{j \in S} |\bX_j ^\top \by| \le o_{\P}(1) D\sqrt{\frac{2k \log p}{n}} + (1 + o_{\P}(1))D\sqrt{\frac{2k \log p}{n}},
\]
which is smaller than $2D\sqrt{2k (\log p)/n}$ with probability tending to one. 
\end{proof}

\begin{proof}[Proof of Lemma \ref{lm:small_proj}]
First, we consider the case where $\|\widehat\bbeta\|_0 \ge  0.25\sqrt{n/\log p}$. Since
\[
\begin{aligned}
\frac{\sqrt{k} \|\widehat{\bbeta}^S\|_0 \log p}{n} \cdot D &\ge \frac{\sqrt{k} \times \sqrt{n/\log p}/4 \times \log p}{n} \cdot D\\
&= \frac14\sqrt{\frac{k\log p}{n}} \cdot D
\end{aligned}
\]
(we suppress the dependence of $\widehat{\bbeta}^S$ on $\lambda$) and
\[
\begin{aligned}
\left| \bX_j^\top \bX \widehat\bbeta^S \right| &\le \left| \bX_j^\top \by \right| + \left| \bX_j^\top (\by - \bX \widehat\bbeta^S) \right|\\
&\le 2\max_i \left| \bX^\top_i \by \right| \\
& \le 4D \sqrt{\frac{2k\log p}{n}}
\end{aligned}
\]
for all $j$, which makes use of Lemma \ref{lm:max_corr} (this also follows from \eqref{eq:lambda_rev} below). In this case, the proof is simply as follows:
\begin{equation}\nonumber
\begin{aligned}
\max_{j:\widehat\beta_j^S(\lambda) = 0} \left| \bX_j^\top \bX \widehat\bbeta^S(\lambda)\right| &\le \max_{1 \le j \le p} \left| \bX_j^\top \bX \widehat\bbeta^S(\lambda)\right|\\
& \le 4D \sqrt{\frac{2k\log p}{n}} \\
& = 16\sqrt{2} \cdot \frac14\sqrt{\frac{k\log p}{n}} \cdot D\\
&\le 16\sqrt{2} \cdot \frac{\sqrt{k} \|\widehat{\bbeta}^S\|_0 \log p}{n} \cdot D.
\end{aligned}
\end{equation}

Now we move to the case where $\|\widehat\bbeta\|_0 \le \sqrt{n/\log p}/4$. Recognizing that the number of pairs $1 \le i< j \le p$ is $p(p-1)/2$, we get
\[
\max_{1 \le i < j \le p} |\bX_j^\top \bX_i | \le (1 + o_{\P}(1)) \sqrt{\frac{2\log \frac{p(p-1)}{2}}{n}} = (1 + o_{\P}(1)) \sqrt{\frac{4\log p}{n}}
\]
under Assumption \ref{ass:working}. This result implies
\begin{equation}\label{eq:inter_p_b}
\begin{aligned}
\left| \bX_j^\top \bX \widehat\bbeta^S\right| &\le \sum_{i} |\widehat\beta^S_i| |\bX_i^\top \bX_j | \\
&=\sum_{i: \widehat\beta^S_i \ne 0} |\widehat\beta^S_i| |\bX_i^\top \bX_j| \\
&\le (1+o_{\P}(1))\sum_{i: \widehat\beta_i^S \ne 0} |\widehat\beta_i^S| \sqrt{\frac{4\log p}{n}} \\
&\le (1+o_{\P}(1)) \|\widehat\bbeta^S\|_0 \max_i |\widehat\beta^S_i| \sqrt{\frac{4\log p}{n}}
\end{aligned}
\end{equation}
for all $j$ such that $\widehat\beta^S_j = 0$. Let $j^\star$ be the index that $|\widehat\beta_{j^\star}^S|$ is the largest. Taking
\begin{equation}\label{eq:small_beta_val}
|\widehat\beta^S_{j^\star}| \le (8 + o_{\P}(1))D \sqrt{\frac{2k \log p}{n}}
\end{equation}
as given for the moment, from \eqref{eq:inter_p_b} we get
\begin{equation}\nonumber
\begin{aligned}
\left| \bX_j^\top \bX \widehat\bbeta^S\right|  &\le (1+o_{\P}(1)) \|\widehat\bbeta^S\|_0 \max_{i} |\widehat\beta_i^S| \sqrt{\frac{4\log p}{n}}\\
&\le (1+o_{\P}(1)) \|\widehat\bbeta^S\|_0 \times 8D \sqrt{\frac{2k \log p}{n}} \times \sqrt{\frac{4\log p}{n}}\\
& = (1+o_{\P}(1)) 16\sqrt{2} \cdot \frac{\sqrt{k} \|\widehat\bbeta^S\|_0 \log p}{n} \cdot D\\
\end{aligned}
\end{equation}
Summarizing from the two cases, the lemma holds if we take any constant $C > 16\sqrt{2}$. 

The rest of the proof aims to verify \eqref{eq:small_beta_val}. Note that, on the one hand,
\begin{equation}\label{eq:lambda_rev}
\begin{aligned}
\lambda &= \left| \bX_{j^\star}^\top (\by - \bX \widehat\bbeta^S)\right| \\
&\le \lambda_{\max} := \max_{1 \le j \le p} |\bX^\top_j \by| \\
& \le 2D \sqrt{\frac{2k \log p}{n}}.
\end{aligned}
\end{equation}
On the other hand, we have
\[
\begin{aligned}
\left| \bX_{j^\star}^\top (\by - \bX \widehat\bbeta^S)\right| & \ge 
\left| \bX_{j^\star}^\top \bX \widehat\bbeta^S \right| - \left| \bX_{j^\star}^\top \by\right|\\
&\ge \left| \bX_{j^\star}^\top \bX \widehat\bbeta^S \right| - \max_{1 \le j \le p} |\bX^\top_j \by|\\
&\ge \|\bX_{j^\star}\|^2|\widehat\beta_{j^\star}^S| - \sum_{i\ne j^\star, \widehat\beta^S_i \ne 0} |\widehat\beta_i^S| |\bX_i^\top \bX_{j^\star}| - \max_{1\le j \le p} |\bX^\top_j \by|\\
&\ge \|\bX_{j^\star}\|^2|\widehat\beta_{j^\star}^S| - \sum_{i \ne j^\star, \widehat\beta_i^S \ne 0} (1 + o_{\P}(1))\sqrt{\frac{4\log p}{n}}  |\widehat\beta_i^S|- \max_{1 \le j \le p} |\bX^\top_j \by|\\
&\ge \|\bX_{j^\star}\|^2|\widehat\beta_{j^\star}^S| - \sum_{i \ne j^\star, \widehat\beta_i^S \ne 0} (1 + o_{\P}(1))\sqrt{\frac{4\log p}{n}}  |\widehat\beta_{j^\star}^S|- \max_{1 \le j \le p} |\bX^\top_j \by|\\
&\ge \left[ \|\bX_{j^\star}\|^2 - (1 + o_{\P}(1))\sqrt{\frac{4\log p}{n}} \|\widehat\bbeta^S\|_0  \right] |\widehat\beta^S_{j^\star}|- \max_{1 \le j \le p} |\bX^\top_j \by|\\
&\ge \left( \frac12 + o_{\P}(1) \right) |\widehat\beta^S_{j^\star}|- \max_{1 \le j \le p} |\bX^\top_j \by|
\end{aligned}
\]
Hence,
\[
\left( \frac12 + o_{\P}(1) \right) |\widehat\beta^S_{j^\star}| \le 2D \sqrt{\frac{2k \log p}{n}} + \max_{1 \le j \le p} |\bX^\top_j \by| \le 4D \sqrt{\frac{2k \log p}{n}}
\]
with probability tending to one. We see that \eqref{eq:small_beta_val} is an immediate consequence.
\end{proof}

\begin{proof}[Proof of Lemma \ref{lm:IJbound}]
For the first case, by definition we have
\[
-\gamma + \sqrt{\frac{2k\log(k/J)}{n}} > \sqrt{\frac{2k \log(p-k)}{n}} - c
\]
and
\[
-\gamma + \sqrt{\frac{2k\log(k/(J+1))}{n}} \le \sqrt{\frac{2k \log(p-k)}{n}} - c,
\]
where we write $J$ for $J_c(\gamma)$. These two inequalities are equivalent to
\begin{equation}\label{eq:eqv_j}
\begin{aligned}
&\log J_c(\gamma) < (c-\gamma)\sqrt{\frac{2n\log(p-k)}{k}} - \frac{(c-\gamma)^2n}{2k} + \log\frac{k}{p-k}\\
&\log(J_c(\gamma)+1) \ge (c-\gamma)\sqrt{\frac{2n\log(p-k)}{k}} - \frac{(c-\gamma)^2n}{2k} + \log\frac{k}{p-k}.
\end{aligned}
\end{equation}
Under Assumption \ref{ass:working}, it is easy to check that
\[
\begin{aligned}
&(c-\gamma)\sqrt{\frac{2n\log(p-k)}{k}} - \frac{(c-\gamma)^2n}{2k} + \log\frac{k}{p-k} \\
&= (1 + o(1))\left[ (c-\gamma)\sqrt{\frac{2n\log p}{k}} - \frac{(c-\gamma)^2n}{2k} + \log\frac{n}{2p\log p} \right]\\
& \goto \infty.
\end{aligned}
\]
Thus, from \eqref{eq:eqv_j} we get
\[
\log J_c(\gamma) = (1 + o(1))\left[ (c-\gamma)\sqrt{\frac{2n\log p}{k}} - \frac{(c-\gamma)^2n}{2k} + \log\frac{n}{2p\log p} \right].
\]
The rest two cases follow from similar reasoning and, thus, their proofs are omitted.

\end{proof}

\begin{proof}[Proof of Lemma \ref{lm:small_proj_stepwise}]

Here, the least-squares estimate is
\[
\widehat{\bm\alpha}^{i,F} = (\bX_F^\top \bX_F)^{-1}\bX_F^\top \bX_i,
\]
which, conditional on $\bX_F$, is distributed as
\[
\N(\bzero, (\bX_F^\top \bX_F)^{-1}/n).
\]
Denote by $\A_m$ the event that
\[
\max_{|F| \le m} \|\bX_F^\top \bX_F - \bI\| \le \frac12,
\]
where $\|\cdot\|$ denotes the matrix spectral norm. Note that we have
\[
\begin{aligned}
\P\left(\max |\widehat\alpha^{i,F}_j| \ge \sqrt{5m(\log p)/n}  \right) &\le \P(\overline\A_m) + \P\left(\max |\widehat\alpha^{i,F}_j| \ge \sqrt{5m(\log p)/n}, \A_m \right)\\
&\le \P(\overline\A_m) + \sum_{i \notin F, j\in F, |F| \le m}\P\left(|\widehat\alpha^{i,F}_j| \ge \sqrt{5m(\log p)/n}, \A_m \right)\\
&\le \P(\overline\A_m) + \sum_{i \notin F, j\in F, |F| \le m}\P\left(|\widehat\alpha^{i,F}_j| \ge \sqrt{5m(\log p)/n}, \|\bX_F^\top \bX_F - \bI\| \le \frac12\right)\\
&\le \P(\overline\A_m) + \sum_{i \notin F, j\in F, |F| \le m}\P\left(|\widehat\alpha^{i,F}_j| \ge \sqrt{5m(\log p)/n} \Big| \|\bX_F^\top \bX_F - \bI\| \le \frac12\right),
\end{aligned}
\]
where the max operator is taken over all triples $(i, j, F)$ such that $|F| \le m, i \notin F$, and $j \in F$. Given $\|\bX_F^\top \bX_F - \bI\| \le 1/2$, all the eigenvalues of $(\bX_F^\top \bX_F)^{-1}$ are upper bounded by 2. Since every diagonal element of a square matrix is lower bounded by the minimum eigenvalue of the matrix, all the diagonal elements of $(\bX_F^\top \bX_F)^{-1}$ are no greater than $2$. As a result,
\[
\P\left(|\widehat\alpha^{i,F}_j| \ge \sqrt{5m(\log p)/n} \Big| \|\bX_F^\top \bX_F - \bI\| \le \frac12\right) \le \P\left( |\mathcal{N}(0,1)| \ge \frac{\sqrt{5m(\log p)/n}}{\sqrt{2/n}} \right) = 2\Phi(-\sqrt{2.5m \log p}).
\]
Hence,
\[
\begin{aligned}
\sum_{i \notin F, j\in F, |F| \le m}\P\left(|\widehat\alpha^{i,F}_j| \ge \sqrt{5m(\log p)/n} \Big| \|\bX_F^\top \bX_F - \bI\| \le \frac12\right) &\le  \sum_{i \notin F, j\in F, |F| \le m} 2\Phi(-\sqrt{2.5m \log p})\\
& \le  p m \cdot \frac{p - m + 1}{p - 2m + 1} {p\choose m} \cdot 2\Phi(-\sqrt{2.5m \log p})\\
& \le  p m \cdot \frac{p - m + 1}{p - 2m + 1} p^m \cdot 2\Phi(-\sqrt{2.5m \log p})\\
&\le p^{m + O(1)}\Phi(-\sqrt{2.5m \log p})\\
&\le p^{m + O(1)}\frac1{\sqrt{2.5m \log p}} \e^{-1.25m\log p}\\
&= \frac1{\sqrt{2.5m \log p}} \e^{-(0.25-o(1)) m\log p}\\
&\goto 0.
\end{aligned}
\]
Together with $\P(\overline \A_m) \goto 0$, which follows from the proof of Theorem \ref{thm:unique}, the last display gives
\[
\P\left(\max |\widehat\alpha^{i,F}_j| \ge \sqrt{5m(\log p)/n}  \right) \goto 0.
\]
To see why $\P(\overline \A_m) \goto 0$, recognize that the term on the right-hand side of \eqref{eq:small_rip_event} is greater than $1/2$ under Assumption \ref{ass:working}.
\end{proof}


\subsection{Theorem \ref{thm:lower}}
\label{sec:proof-theor-refthm:l}
As earlier in Section \ref{sec:proof-theor-refthm:l-1}, we first state some lemmas before turning to the proof of Theorem \ref{thm:lower}.

\begin{lemma}\label{lm:dicho}
Fit the response $\by$ on the true support $\bX_S$ using the lasso and denote by $\widehat{\bbeta}^S(\lambda)$ the lasso solution. 
Under Assumption \ref{ass:working}, further assume that $0.5 < \Gamma < 1.1$ with probability tending to one. Then, for an arbitrary constant $c > 0$, with probability approaching one, all the variables $\bX_{(1)}, \bX_{(2)}, \ldots, \bX_{(I_{-c})}$ enter the model before any of 
\[
\bX_{(I_c+1)}, \bX_{(I_c+2)}, \ldots, \bX_{(k)}
\]
along the lasso path. 
\end{lemma}

\begin{lemma}\label{lm:gamma_1}
Under Assumption \ref{ass:working}, further assume
\[
\frac{\sigma}{M} \sqrt{\frac{n}{k}} \goto 0.
\]
Then, $\Gamma = 1 + o_{\P}(1)$.
\end{lemma}

With Lemmas \ref{lm:dicho} and \ref{lm:gamma_1} in place, now we give the proof of Theorem \ref{thm:lower} where the procedure is the lasso or least angle regression. Note that for large $\gamma$ and sufficiently small $c$, we have $J_c(\gamma) = 0$.
\begin{proof}[Proof of Theorem \ref{thm:lower} in the lasso and least angle regression cases]
The proof starts by making use of Lemma \ref{lm:gamma_1}, which concludes
\[
1 - c \le \Gamma \le 1 + c
\]
with probability tending to one for any constant $c > 0$. To see this, note that Assumption \ref{ass:working} along ensures $\sqrt{n/k} \le 1/\sqrt{c_4} = O(1)$ and, as a result, we get
\[
\frac{\sigma}{M} \sqrt{\frac{n}{k}} \goto 0
\]
given the condition that $\sigma/M \goto 0$.

As in the proof of Theorem \ref{thm:upper}, here it also suffices to only focus on the lasso case. Consider fitting the lasso on $\bX_S$. Denote by $\bX_{(I^\diamond)}$ the last variable among $\bX_{(1)}, \bX_{(2)}, \ldots, \bX_{(I_{-c})}$ that enters the model and let $\lambda^\diamond$ be the lasso penalty when $\bX_{(I^\diamond)}$ is just about to enter the lasso path. Note that $\widehat\beta^S_{(I^\diamond)}(\lambda^\diamond) = 0$ and $I_{-c} = I_{-c}(\Gamma)$. The KKT conditions give
\begin{equation}\label{eq:kkt_lmb_di}
\left|\bX_{(I^\diamond)}^\top \left( \by - \bX \widehat\bbeta^S \right) \right|= \lambda^\diamond.
\end{equation}
Note that
\begin{equation}\label{eq:d_d_bb}
\begin{aligned}
\bX_{(I^\diamond)}^\top \by &\ge  \bX_{(I_{-c})}^\top \by  = D \left[\Gamma + \sqrt{\frac{2k\log(k/(I_{-c}))}{n}} + o_{\P}(1) - \frac{\xi}{D}\right]\\
&= D \left[\Gamma + \sqrt{\frac{2k\log(k/(I_{-c}))}{n}} + o_{\P}(1)\right]\\
&= D \left[\sqrt{\frac{2k\log(p-k)}{n}} + c + o_{\P}(1) + o_{\P}(1)\right]\\
&= D \left[\sqrt{\frac{2k\log(p-k)}{n}} + c + o_{\P}(1)\right].
\end{aligned}
\end{equation}
From Lemma \ref{lm:dicho} we see that, with probability approaching one, by the time of $\lambda^\diamond$ none of the variables $\bX_{(I_c+1)}, \bX_{(I_c+2)}, \ldots, \bX_{(k)}$ has been included in the lasso path. Hence, the lasso model consists of no more than $I_c$ variables at any time before $\lambda^\diamond$, and it is easy to check that 
\[
\|\lasso^S(\lambda)\|_0 \le I_c(\Gamma) \le o_{\P}\left(\frac{n}{\sqrt{k} \log p} \right)
\]
for all $\lambda > \lambda^\diamond$. Consequently, we get
\[
\frac{C\sqrt{k} \|\lasso^S(\lambda)\|_0 \log p}{n} \goto 0
\]
uniformly for all $\lambda > \lambda^\diamond$. Thus, Lemma \ref{lm:small_proj} shows $|\bX_{(I^\diamond)}^\top\bX \widehat\bbeta^S| = o_{\P}(D)$. Hence, taking \eqref{eq:kkt_lmb_di} and \eqref{eq:d_d_bb} together shows that $\lambda^\diamond$ obeys
\begin{equation}\label{eq:diamond_ine}
\lambda^\diamond \ge D \left[\sqrt{\frac{2k\log(p-k)}{n}} + c + o_{\P}(1)\right].
\end{equation}
(Similarly, given a small $I_c$, Theorem \ref{thm:unique} ensures that by that time no drop-out has ever happened and, consequently, the lasso and least angle regression are equivalent in our discussion.)
If we can show that
\begin{equation}\label{eq:resi_lasso_sm}
\left|\bX_j^\top \left( \by - \bX \widehat\bbeta^S(\lambda) \right) \right| < \lambda
\end{equation}
for all $\lambda > \lambda^\diamond$ and all $j \notin S$, then the full lasso (on $\bX$) selects variables only from $\bX_S$ before time $\lambda^\diamond$. Call this event $\mathcal{B}$. Thus, on $\mathcal{B}$ the first $I_{-c}$ selected variables along the lasso path are all signal variables. As for $I_{-c}$, Lemma \ref{lm:IJbound} yields
\[
\begin{aligned}
\log T &\ge \log (I_{-c}(\Gamma) + 1)\\
&= (1+o_{\P}(1)) \left[ (-c+\Gamma)\sqrt{\frac{2n\log(p-k)}{k}} - \frac{(-c+\Gamma)^2n}{2k} + \log \frac{n}{2p\log p}\right].
\end{aligned}
\]
Provided $1 - c \le \Gamma \le 1 + c$ with probability tending to one, we get
\[
(-c+\Gamma)\sqrt{\frac{2n\log(p-k)}{k}} - \frac{(-c+\Gamma)^2n}{2k} \ge (1 - 2c)\sqrt{\frac{2n\log(p-k)}{k}} - \frac{(1 - 2c)^2n}{2k}.
\]
Hence, we have
\[
\log T \ge (1 + o_{\P}(1))\left[ (1 - 2c )\sqrt{\frac{2n\log(p-k)}{k}} - \frac{(1 - 2c)^2n}{2k} + \log\frac{n}{2p\log p}\right]
\]
for all constant $c > 0$. Setting $c \goto 0+$ gives
\[
\begin{aligned}
\log T &\ge (1 + o_{\P}(1))\left[ \sqrt{\frac{2n\log(p-k)}{k}} - \frac{n}{2k} + \log\frac{n}{2p\log p}\right] \\
&= (1 + o_{\P}(1))\left[ \sqrt{\frac{2n\log p}{k}} - \frac{n}{2k} + \log\frac{n}{2p\log p}\right],
\end{aligned}
\]
as desired.

Establishing \eqref{eq:resi_lasso_sm} is the subject of the remaining proof. That is, prove $\P(\mathcal B) \goto 1$. For $j \notin S$, observe that
\begin{equation}\label{eq:kkt_twoterm}
\left|\bX_j^\top \left( \by - \bX \widehat\bbeta^S\right) \right| \le |\bX_j^\top \by| + |\bX_j^\top \bX \widehat\bbeta^S|.
\end{equation}
Making use of the independence between $\by$ and $\bX_j$, the first term $\bX_j ^\top \by$ obeys
\[
\max_{j \notin S} |\bX_j ^\top \by| \le \|\by\|\sqrt{\frac{2\log(p-k)}{n}}
\]
with probability approaching one. The second term, $|\bX_j^\top \bX \widehat\bbeta^S|$, as shown earlier by Lemma \ref{lm:small_proj}, satisfies that
\[
\max_{j \notin S, \lambda > \lambda^\diamond} |\bX_j^\top \bX \widehat\bbeta^S(\lambda)| \le \frac{C \sqrt{k} \|\widehat\bbeta^S\|_0\log p}{n} \cdot D = o_{\P}(D)
\]
with probability tending to one. Hence, \eqref{eq:kkt_twoterm} yields
\[
\begin{aligned}
\left|\bX_j^\top \left( \by - \bX \widehat\bbeta^S\right) \right| &\le \|\by\|\sqrt{\frac{2\log(p-k)}{n}} + o_{\P}(D) \\
&= D\sqrt{\frac{2k\log(p-k)}{n}} + o_{\P}(D) \\
& = D\left[\sqrt{\frac{2k\log(p-k)}{n}} + o_{\P}(1) \right]\\
& < D \left[\sqrt{\frac{2k\log(p-k)}{n}} + c + o_{\P}(1)\right]\\
& \le \lambda^\diamond\\
& < \lambda,
\end{aligned}
\]
where the second last inequality follows from \eqref{eq:diamond_ine}. Therefore, $\P(\mathcal B) \goto 1$.
\end{proof}

Next, we turn to prove Theorem \ref{thm:lower} in the case of forward stepwise regression.
\begin{lemma}\label{lm:dicho_fs}
Fit the response $\by$ on the true support $\bX_S$ using forward stepwise. Under Assumption \ref{ass:working}, further assume that $0.5 < \Gamma < 1.1$ with probability tending to one. Then, for an arbitrary constant $c > 0$, with probability approaching one, all the variables $\bX_{(1)}, \bX_{(2)}, \ldots, \bX_{(I_{-2c})}$ enter the model before any of 
\[
\bX_{(I_{-c}+1)}, \bX_{(I_{-c}+2)}, \ldots, \bX_{(k)}
\]
by forward stepwise.

\end{lemma}

\begin{proof}[Proof of Theorem \ref{thm:lower} in the forward stepwise case]
Consider running forward stepwise on the design matrix $\bX_S$. Denote by $\bX_{I^\diamond}$ the last variable among $\bX_{(1)}, \ldots, \bX_{(I_{-2c})}$ that gets selected. Let $m^\star - 1$ be the number of variables selected prior to $\bX_{(I^\diamond)}$. From Lemma \ref{lm:dicho_fs} we see that, with probability approaching to one, by then none of $\bX_{(I_{-c}+1)}, \ldots, \bX_{(k)}$ has been selected, that is, $m^\star \le I_{-c}$. The proof would be completed once we show that by then no noise variables would be selected if we perform forward stepwise on $\bX$ instead of $\bX_S$. 

Denote by $F_m$ the set of variables selected in the first $m$ steps. We would like to show that
\begin{equation}\label{eq:fs_on_s}
\min_{\supp(\bb) = \fsset_{m}}\| \by - \bX b\|^2 \le \min_{\supp(\bb) = \fsset_{m-1} \cup l, l \in \overline S}\| \by - \bX b\|^2
\end{equation}
for all $m \le m^\star$. Denote by $\bX_{(j_m)}$ the variable selected in the $m$th step. To this end, note that from the proof of Theorem \ref{thm:upper} for the forward stepwise case, we see that for any fixed $l \in \overline{S}$,
\[
\begin{aligned}
&\min_{\supp(\bb) = \fsset_{m-1} \cup l}\| \by - \bX b\|^2 - \min_{\supp(\bb) = \fsset_m}\| \by - \bX b\|^2 \\
&= (\bX^\top_{(j_m)} \by + o_{\P}(D))^2 - (\bX^\top_l \by + o_{\P}(D))^2.
\end{aligned}
\]
Note that since $l \in \overline{S}$, we have $|\bX^\top_l \by| \le D(\sqrt{2k\log(p-k)/n} + o_{\P}(1))$, and since $1 \le j_m \le I_{-c}$, we get
\[
\bX^\top_{(j_m)} \by \ge \bX^\top_{(I_{-c})} \by = D \left[\sqrt{\frac{2k\log(p-k)}{n}} + c + o_{\P}(1)\right].
\]
An immediate consequence is the following:
\[
\begin{aligned}
&\min_{\supp(\bb) = \fsset_{m-1} \cup l}\| \by - \bX b\|^2 - \min_{\supp(\bb) = \fsset_m}\| \by - \bX b\|^2\\
& \ge D^2 \left[\sqrt{\frac{2k\log(p-k)}{n}} + c + o_{\P}(1)\right]^2 - D^2 \left[\sqrt{\frac{2k\log(p-k)}{n}}  + o_{\P}(1)\right]^2\\
& > 0,
\end{aligned}
\]
which certifies \eqref{eq:fs_on_s} provided the arbitrariness of $l \in \overline S$. Therefore, with probability approaching one, we get
\[
\log T \ge \log(I_{-2c}(\Gamma) + 1).
\]
Setting $c \goto 0+$ completes the proof. Details are the same as the proof for the case of the lasso and least angle regression.

\end{proof}

To conclude this section, we prove Lemmas \ref{lm:dicho}, \ref{lm:gamma_1}, and \ref{lm:dicho_fs}.
\begin{proof}[Proof of Lemma \ref{lm:dicho}]
Consider the first time along the lasso path that a variable among $\bX_{(I_c+1)}, \bX_{(I_c+2)}, \ldots, \bX_{(k )}$ is just about to enter the lasso model and denote by $\bX_{(L)}$ this variable, where $I_c +1 \le L \le k$. Specifically, writing $\lambda^\ast$ for the lasso penalty at this point, we know that
\[
\widehat\beta_{(L)}(\lambda^\ast) = 0
\]
and
\[
\widehat\beta_{(L)}(\lambda) \ne 0
\]
if $\lambda^\ast - c' < \lambda < \lambda^\ast$ for some $c' > 0$. For a proof of this lemma by contradiction, assume that $\bX_{(l)}$ has not yet entered the model at $\lambda^\ast$ for some $l$ satisfying $1 \le l \le I_{-c}$.  Under this assumption, our discussion below considers the lasso solution $\widehat\bbeta$ at $\lambda^\ast$. First of all, the KKT conditions of the lasso give
\[
\left| \bX_{(l)}^\top \left( \by - \bX \widehat\bbeta \right)\right| \le \lambda^\ast, \quad  \left|\bX_{(L)}^\top \left( \by - \bX \widehat\bbeta \right) \right|= \lambda^\ast.
\]
Consequently,
\begin{equation}\label{eq:kkt_compare}
\left| \bX_{(l)}^\top \left( \by - \bX \widehat\bbeta \right)\right| \le \left| \bX_{(L)}^\top \left( \by - \bX \widehat\bbeta \right)\right|.
\end{equation}
Thus, the proof of the present lemma would be finished if we show that, with probability tending to one, \eqref{eq:kkt_compare} cannot be satisfied. 

The rest part of the proof is devoted to disproving \eqref{eq:kkt_compare}. Note that $\bX_{(l)}^\top \by$ satisfies
\[
\begin{aligned}
\bX_{(l)}^\top \by \ge \bX_{(I_{-c})}^\top \by 
&= D \left[\Gamma + \sqrt{\frac{2k\log(k/I_{-c})}{n}} + o_{\P}(1) - \frac{\xi}{D} \right]\\
&= D \left[\Gamma + \sqrt{\frac{2k\log(k/I_{-c})}{n}} + o_{\P}(1)\right]\\
&= D \left[\sqrt{\frac{2k\log(p-k)}{n}} + o_{\P}(1)\right],
\end{aligned}
\]
which, together with Lemma \ref{lm:max_corr}, implies that
\begin{equation}\label{eq:l_boundA}
  \begin{aligned}
\left| \bX_{(l)}^\top \left( \by - \bX \widehat\bbeta \right) \right| &\ge D \left[\sqrt{\frac{2k\log(p-k)}{n}} + o_{\P}(1) - |\bX_{(l)}^\top \bX \widehat\bbeta|/D \right] \\
&= D \left[\sqrt{\frac{2k\log(p-k)}{n}} + o_{\P}(1)\right].
  \end{aligned}
\end{equation}
Similarly, from $I_c +1 \le L \le k$ we get
\begin{equation}\label{eq:L_sandwichA}
\bX_{(I_c+1)}^\top \by \ge \bX_{(L)}^\top \by \ge \bX_{(k)}^\top \by.
\end{equation}
The quantities appearing on the left- and right-hand sides obey, respectively,
\[
\bX_{(I_c+1)}^\top \by =  D \left[\sqrt{\frac{2k\log(p-k)}{n}} - c + o_{\P}(1)\right]
\]
and
\[
\bX_{(k)}^\top \by = -D \left[\sqrt{\frac{2k\log(p-k)}{n}} - \Gamma + o_{\P}(1)\right].
\]
Hence, from \eqref{eq:L_sandwichA} it follows that
\[
D \left[\sqrt{\frac{2k\log(p-k)}{n}} - c + o_{\P}(1)\right] \ge \bX_{(L)}^\top \by \ge -D \left[\sqrt{\frac{2k\log(p-k)}{n}} - \Gamma + o_{\P}(1)\right],
\]
yielding
\begin{equation}\label{eq:L_boundA}
\begin{aligned}
\left| \bX_{(L)}^\top \left( \by - \bX \widehat\bbeta \right)\right| &\le D \left[\sqrt{\frac{2k\log(p-k)}{n}} - c + o_{\P}(1) + |\bX_{(L)}^\top \bX \widehat\bbeta|/D\right]\\
&= D \left[\sqrt{\frac{2k\log(p-k)}{n}} - c + o_{\P}(1)\right]
\end{aligned}
\end{equation}

Last, combining \eqref{eq:l_boundA} and \eqref{eq:L_boundA} shows that, with probability tending to one, \eqref{eq:kkt_compare} cannot be satisfied, as desired.

\end{proof}

\begin{proof}[Proof of Lemma \ref{lm:gamma_1}]
Recall that the definition
\[
\Gamma = \frac{(\bX \bbeta)^\top \by}{\sqrt{k} M \|\by\|}.
\]
Observing that $\|\by\| = (1 + o_{\P}(1)) \sqrt{kM^2 + n\sigma^2}$, we get
\[
\begin{aligned}
\Gamma &= \frac{(\bX \bbeta)^\top \by}{\sqrt{k} M (1 + o_{\P}(1)) \sqrt{kM^2 + n\sigma^2}}\\
& = (1 + o_{\P}(1))\left\langle \frac{\bX \bbeta}{\sqrt{k} M}, ~ \frac{\by}{\sqrt{kM^2 + n\sigma^2}} \right\rangle.
\end{aligned}
\]
Hence, the present lemma is equivalent to
\begin{equation}\label{eq:gamma_tilt}
\left\langle \frac{\bX \bbeta}{\sqrt{k} M}, ~ \frac{\by}{\sqrt{kM^2 + n\sigma^2}} \right\rangle = 1 + o_{\P}(1),
\end{equation}
where $\langle, \rangle$ denotes the usual inter product of vectors. Since
\[
\frac{\bX\bbeta}{\sqrt{k} M} \text{ and } \frac{\by}{\sqrt{kM^2 + n\sigma^2}}
\]
are vectors of unit norm asymptotically, \eqref{eq:gamma_tilt} boils down to claiming that the angle between the two vectors are asymptotically zero. Recognizing that $\bX\bbeta$ and $\bz$ are asymptotically orthogonal since they are independent high-dimensional normal vectors, a vanishing angle between $\bX\bbeta$ and $\by = \bX\bbeta + \bz$ is equivalent to
\[
\frac{\|\bz\|}{\|\bX \bbeta\|} \goto 0.
\]
The display above is a direct consequence of the condition
\[
\frac{\sqrt{n}\sigma}{\sqrt{k} M} \goto 0,
\]
which is provided in the assumptions. Hence, \eqref{eq:gamma_tilt} holds.
\end{proof}

\begin{proof}[Proof of Lemma \ref{lm:dicho_fs}]
Consider the first time a variable among $\bX_{(I_{-c}+1)}, \ldots, \bX_{(k)}$ enters the model. Denote by $m^\diamond$ the rank of this variable. As earlier, call this variable $\bX_{(j_{m^\diamond})}$. Note that, by definition, $I_{-c} + 1 \le j_{m^\diamond} \le k$. Suppose on the contrary that by the time $\bX_{(j_{m^\diamond})}$ is selected, at least one variable among $\bX_{(1)}, \ldots, \bX_{(I_{-2c})}$, say $\bX_{(l)}$, has not been included. Denote by $\C$ this event, on which we must have
\begin{equation}\label{eq:a_event_holds}
\min_{\supp(\bb) = \fsset_{m^\diamond}}\| \by - \bX b\|^2 \le \min_{\supp(\bb) = \fsset_{m^\diamond-1} \cup (l)}\| \by - \bX b\|^2.
\end{equation}
From the proof of Theorem \ref{thm:upper} for the forward stepwise case, we know that
\begin{equation}\label{eq:fs_part}
\begin{aligned}
&\min_{\supp(\bb) = \fsset_{m^\diamond-1} \cup (l)}\| \by - \bX b\|^2 - \min_{\supp(\bb) = \fsset_m^\diamond}\| \by - \bX b\|^2 \\
&= (\bX^\top_{(j_{m^\diamond})} \by + o_{\P}(D))^2 - (\bX^\top_{(l)} \by + o_{\P}(D))^2.
\end{aligned}
\end{equation}
Since $j_{m^\diamond} \ge I_{-c} + 1$, we get
\begin{equation}\label{eq:jm1}
\bX^\top_{(j_{m^\diamond})} \by \le \bX^\top_{(I_{-c})} \by = D \left[\sqrt{\frac{2k\log(p-k)}{n}} + c + o_{\P}(1)\right],
\end{equation}
and from $j_{m^\diamond} \le k$ we get
\begin{equation}\label{eq:jm2}
\begin{aligned}
\bX^\top_{(j_{m^\diamond})} &\by \ge \bX^\top_{(k)} \by \\
&= D \left[\Gamma - \sqrt{\frac{2k\log k}{n}} + o_{\P}(1)\right]\\
& > -D \left[\sqrt{\frac{2k\log(p-k)}{n}} + c + o_{\P}(1)\right],
\end{aligned}
\end{equation}
which makes use of the fact that $0.5 < \Gamma < 1.1$.
Similarly, we have
\begin{equation}\label{eq:jm3}
\bX^\top_{(l)} \by \ge \bX^\top_{(I_{-2c})} \by \ge D \left[\sqrt{\frac{2k\log(p-k)}{n}} + 2c + o_{\P}(1)\right].
\end{equation}
Plugging \eqref{eq:jm1}, \eqref{eq:jm2}, and \eqref{eq:jm3} into \eqref{eq:fs_part} yields
\[
\begin{aligned}
&\min_{\supp(\bb) = \fsset_{m^\diamond-1} \cup (l)}\| \by - \bX b\|^2 - \min_{\supp(\bb) = \fsset_m^\diamond}\| \by - \bX b\|^2 \\
& \le D^2 \left[\sqrt{\frac{2k\log(p-k)}{n}} + c + o_{\P}(1)\right]^2 - D^2 \left[\sqrt{\frac{2k\log(p-k)}{n}} + 2c + o_{\P}(1)\right]^2\\
& < 0,
\end{aligned}
\]
which contradicts \eqref{eq:a_event_holds}. This implies the event $\C$ happens with probability vanishing to zero.

\end{proof}


\subsection{Theorem \ref{thm:unique}}
\label{sec:proof-theor-refthm-1}

\begin{proof}[Proof of Theorem \ref{thm:unique}]
Write
\[
m = \min\left\{ \left\lceil c\sqrt{n/\log p} \right\rceil, p \right\}
\]
for some constant $c$ to be determined later. A variable might be dropped by the lasso only if its fitted coefficient crosses zero. By using this fact, it suffices to show that the fitted coefficients of the selected variables shall never cross zero in the first $m$ selected variables, except for a rare event with probability no more than $1/p^2$. 

Let $\lambda'$ be the first time along the lasso path a previously selected variable is just about to drop out of the model. (The proof shall only focus on the event that such $\lambda'$ exists; otherwise no variable drops out.) Denote by $j$ the number of variables selected just before the first dropout and $\wS$ the set of these $j$ variables. For the sake of contradiction, assume that $j \le m - 1$.

Pick an $\lambda$ that is (slightly) smaller than $\lambda'$ and at which $j$ variables have been included in the model. Observe the following partial KKT condition for the lasso solution:
\begin{equation}\label{eq:lasso_kkt}
-\bX^\top_{\wS} (\by - \bX_{\wS} \lasso_{\wS}) + \lambda\sgn(\lasso_{\wS}) = 0.
\end{equation}
In a componentwise manner, $\sgn(\cdot)$ above returns 1 if the corresponding component is positive and $-1$ if negative. From \eqref{eq:lasso_kkt} it follows that
\begin{equation}\nonumber
\lasso_{\wS}(\lambda) = (\bX^\top_{\wS}\bX_{\wS})^{-1} \bX^\top_{\wS} \by - \lambda (\bX^\top_{\wS}\bX_{\wS})^{-1}\sgn(\lasso_{\wS}(\lambda)).
\end{equation}
Now let us gradually decrease $\lambda$ to $\lambda'$, moving along the lasso path. If $(\bX^\top_{\wS}\bX_{\wS})^{-1}\sgn(\lasso_{\wS})$ has the same sign as $\sgn(\lasso_{\wS})$ across all $|\wS| = j$ coordinates, then in this process of moving $\lambda$ down to $\lambda'$, we see all positive coefficients of $\lasso(\lambda)$ get larger while all negative coefficients become even smaller. This implies that no coefficient will cross zero, a contradiction to the assumption. Therefore, $(\bX^\top_{\wS}\bX_{\wS})^{-1}\sgn(\lasso_{\wS})$ differs from $\sgn(\lasso_{\wS})$ in the sign of at least one coordinate, yielding
\begin{equation}\label{eq:sign_diff_b}
\left\| (\bX^\top_{\wS}\bX_{\wS})^{-1}\sgn(\lasso_{\wS}) - \sgn(\lasso_{\wS}) \right\| \ge 1.
\end{equation}
On the other hand, we have
\begin{equation}\label{eq:use_rip}
\begin{aligned}
\left\| (\bX^\top_{\wS}\bX_{\wS})^{-1}\sgn(\lasso_{\wS}) - \sgn(\lasso_{\wS}) \right\| &\le \left\| (\bX^\top_{\wS}\bX_{\wS})^{-1} - \bm I \right\| \left\| \sgn(\lasso_{\wS}) \right\| \\
&\le  \left\| (\bX^\top_{\wS}\bX_{\wS})^{-1} - \bm I \right\| \sqrt{m-1}\\
&\le  \left\| \bX^\top_{\wS}\bX_{\wS} - \bm I \right\| \left\| (\bX^\top_{\wS}\bX_{\wS})^{-1} \right\| \sqrt{m-1}. 
\end{aligned}
\end{equation}
Let $\theta \in (0, 1)$ be the restricted isometry constant for $(m-1)$-sparse vectors, which is defined as the smallest $\theta$ such that
\begin{equation}\nonumber
(1 - \theta) \|\bb\|^2 \le \|\bX \bb\|^2 \le (1 + \theta) \|\bb\|^2
\end{equation}
for all $(m-1)$-sparse $\bb \in \R^p$ (we use the notation $\theta$ instead of the conventional $\delta$ since the latter has been reserved for denoting the ratio $n/p$). It is not hard to see that this definition yields the following more amenable expression
\begin{equation}\label{eq:rip}
\theta = \max_{|U| \le m-1} \|\bX^\top_U \bX_U - \bm I\|,
\end{equation}
where the maximum is taken over all subsets of $\{1, 2, \ldots, p\}$ with cardinality smaller than $m$ and $\bX^\top_U$ denotes $(\bX_U)^\top$. In the case where $m \ge 2$ (if $m = 1$ then this theorem is trivially true), a well known result regarding this constant states
\begin{equation}\label{eq:small_rip_event}
\theta \le C\sqrt{\frac{(m-1)\log(p/(m-1))}{n}}
\end{equation}
with probability at least $1 - 1/p^2$, where $C > 0$ is a universal constant (see, e.g., Theorem 5.2 in \cite{baraniuk2008simple}; note that $\theta$ is random due to its dependence on $\bX$).

Now we prove that on the event \eqref{eq:small_rip_event}, the assumption $j \le m - 1$ cannot hold with a suitable of the constant $c$. To this end, we start by observing that \eqref{eq:rip} together with the assumption $|\wS| = j \le m - 1$ ensures that all the eigenvalues of $\bX^\top_{\wS}\bX_{\wS}$ are between $1 - \theta$ and $1 + \theta$. Thus, from \eqref{eq:use_rip} we get
\[
\begin{aligned}
\left\| (\bX^\top_{\wS}\bX_{\wS})^{-1}\sgn(\lasso_{\wS}) - \sgn(\lasso_{\wS}) \right\| &\le \left\| \bX^\top_{\wS}\bX_{\wS} - \bm I \right\| \left\| (\bX^\top_{\wS}\bX_{\wS})^{-1} \right\| \sqrt{m-1}\\
&\le \frac{\theta}{1 - \theta}\sqrt{m-1},
\end{aligned}
\]
Substituting \eqref{eq:sign_diff_b} into the above display, under the assumption that $j \le m - 1$ we see that
\begin{equation}\label{eq:want_wrong}
\frac{C\sqrt{\frac{(m-1)\log(p/(m-1))}{n}}}{ 1- C\sqrt{\frac{(m-1)\log(p/(m-1))}{n}}} \cdot \sqrt{m-1} \ge \frac{\theta}{1 - \theta} \sqrt{m-1} \ge 1
\end{equation}
holds on the event \eqref{eq:small_rip_event}. If we can show that
\begin{equation}\label{eq:very_wrong}
\frac{C\sqrt{\frac{(m-1)\log(p/(m-1))}{n}}}{ 1- C\sqrt{\frac{(m-1)\log(p/(m-1))}{n}}} \cdot \sqrt{m-1} < 1,
\end{equation}
meaning that \eqref{eq:want_wrong} cannot be satisfied, then the assumption that $j \le m -1$ must be violated and, consequently, the first $m$ variables are included along the lasso path without any drop-out with probability at least $1 - 1/p^2$ (note that the event \eqref{eq:small_rip_event} happens with probability at least $1 - 1/p^2$). Indeed, \eqref{eq:very_wrong} is true if we set (note that we have excluded the trivial case $n < \log p$)
\[
c = \min\left\{1, \frac1{4C^2}, \frac1{2C}\right\}.
\]
This concludes the proof.
\end{proof}


\subsection{Proposition \ref{prop:diagram}}
\label{sec:proof-theorem-4}

\begin{proof}[Proof of Proposition \ref{prop:diagram}]
Note that the least-squares estimator takes the form:
\[
\begin{aligned}
\widehat{\bbeta}^{\textnormal{LS}} &= (\bX^\top \bX)^{-1} \bX^\top\by\\
&= \bbeta + (\bX^\top \bX)^{-1} \bX^\top\bz,
\end{aligned}
\]
which is, conditional on $\bX$, distributed as
\[
\frac{\widehat{\beta}^{\textnormal{LS}}_j}{\sigma\sqrt{[(\bX^\top \bX)^{-1}]_{jj}}} \sim \N(0,1)
\]
for all $j \notin S$ and
\[
\frac{\widehat{\beta}^{\textnormal{LS}}_j}{\sigma\sqrt{[(\bX^\top \bX)^{-1}]_{jj}}} \sim \N(\mu_j,1)
\]
for $j \in S$. Above,
\[
\mu_j = \frac{M}{\sigma\sqrt{[(\bX^\top \bX)^{-1}]_{jj}}}.
\]

Take the following result as given for the moment:
\begin{equation}\label{eq:diag}
\max_{1 \le j \le p} \left[ (\bX^\top \bX)^{-1} \right]_{jj} \le \frac{n}{n - p} + c
\end{equation}
with probability approaching one for an arbitrary constant $c > 0$. Then, we can show that
\begin{equation}\label{eq:min_max}
\min_{j \in S} \left| \frac{\widehat{\beta}^{\textnormal{LS}}_j}{\sigma\sqrt{[(\bX^\top \bX)^{-1}]_{jj}}} \right| > \max_{j \notin S}\left|\frac{\widehat{\beta}^{\textnormal{LS}}_j}{\sigma\sqrt{[(\bX^\top \bX)^{-1}]_{jj}}} \right|.
\end{equation}
To see this, note that
\[
\min_{j \in S} \left| \frac{\widehat{\beta}^{\textnormal{LS}}_j}{\sigma\sqrt{[(\bX^\top \bX)^{-1}]_{jj}}} \right| 
\ge \min_{j \in S} \frac{M}{\sigma\sqrt{[(\bX^\top \bX)^{-1}]_{jj}}} - \sqrt{2\log k}
\]
and
\[
\max_{j \notin S} \left| \frac{\widehat{\beta}^{\textnormal{LS}}_j}{\sigma\sqrt{[(\bX^\top \bX)^{-1}]_{jj}}} \right| \le \sqrt{2\log(p-k)}
\]
with probability approaching one. Hence, \eqref{eq:min_max} follows if one can show that
\[
\min_{j \in S} \frac{M}{\sigma\sqrt{[(\bX^\top \bX)^{-1}]_{jj}}} - \sqrt{2\log k} -  \sqrt{2\log(p-k)} > 0.
\]
In fact, from \eqref{eq:diag} we have
\[
\begin{aligned}
&\min_{j \in S} \frac{M}{\sigma\sqrt{[(\bX^\top \bX)^{-1}]_{jj}}} - \sqrt{2\log k}  -  \sqrt{2\log(p-k)} \\
& > \min_{j \in S} \frac{M}{\sigma\sqrt{[(\bX^\top \bX)^{-1}]_{jj}}} - 2\sqrt{2\log p} \\
& \ge \frac{M}{\sigma\sqrt{\frac{n}{n - p} + c}} - 2\sqrt{2\log p}\\
& \ge \frac{M}{\sigma\sqrt{\frac{\delta}{\delta - 1} + c}} - 2\sqrt{2\log p}\\
& \ge \frac{3\sqrt{\frac{2\delta\log p}{\delta - 1}}}{\sqrt{\frac{\delta}{\delta - 1} + c}} - 2\sqrt{2\log p}\\
& = 3\sqrt{2\log p}  \cdot \sqrt{\frac{\delta/(\delta - 1)}{\delta/(\delta-1) + c}} - 2\sqrt{2\log p},
\end{aligned}
\]
which is positive by setting $c$ sufficiently small. Above, recall that $\delta = n/p$. Hence, \eqref{eq:min_max} holds, meaning that all the true variables are vertically ranked higher than any of the false variables. In particular, the first false variable is vertically ranked lower than all of the true variables.

In the remaining part of the proof, our aim is to prove \eqref{eq:diag}. Write the singular value decomposition of $\bX$ as
\[
\bX = \bU \bD \bV^\top,
\]
where $\bU \in \R^{n \times n}$ and $\bV \in \R^{p \times p}$ are orthogonal matrices, and $\bD \in \R^{n \times p}$ is diagonal. Then, the left-hand side of \eqref{eq:diag} can be expressed as
\[
\left[ (\bX^\top \bX)^{-1} \right]_{jj} = \be_j^\top \bV (\bD^\top \bD)^{-1} V^\top \be_j,
\]
where $\be_j$ denotes the $j$th canonical basis vector in Euclidean space. Using this fact, we get
\[
\begin{aligned}
&\P\left( \max_{1 \le j \le p} \left[ (\bX^\top \bX)^{-1} \right]_{jj} \ge \frac{n}{n - p} + c \right)\\
&  = \P\left( \max_{1 \le j \le p} \be_j^\top \bV (\bD^\top \bD)^{-1} V^\top \be_j \ge \frac{n}{n - p} + c \right)\\
&  = \E \P\left( \max_{1 \le j \le p} \be_j^\top \bV (\bD^\top \bD)^{-1} V^\top \be_j \ge \frac{n}{n - p} + c \Big | \bD\right)\\
&  \le \E \min\left\{  \sum_{j=1}^p \P\left( \be_j^\top \bV (\bD^\top \bD)^{-1} V^\top \be_j \ge \frac{n}{n - p} + c \Big | \bD\right), 1 \right\}\\
&= \E \min\left\{  p\P\left( \be_1^\top \bV (\bD^\top \bD)^{-1} V^\top \be_1 \ge \frac{n}{n - p} + c \Big | \bD\right), 1 \right\},
\end{aligned}
\]
where the last step makes use of the exchangeability of $\bV^\top \be_1, \ldots, \bV^\top \be_p$ given $\bD$. Therefore, writing $\bm{\eta}  = \bV^\top \be_1$, it suffices to prove that
\begin{equation}\nonumber
\E \min\left\{  p\P\left( \bm{\eta}^\top (\bD^\top \bD)^{-1} \bm{\eta} \ge \frac{n}{n - p} + c \Big | \bD\right), 1 \right\} \goto 0.
\end{equation}

Recognizing that $\bX$ has independent $\N(0, 1/n)$ entries, it is known that $\bm\eta$ is independent of $\bD$ and is uniformly distributed on the unit sphere in $\R^p$. In particular, $\bm\eta$ assumes the following representation
\[
\bm\eta = \frac{(\zeta_1, \ldots, \zeta_p)^\top}{\sqrt{\zeta_1^2 + \cdots + \zeta_p^2}} = \frac{\bm{\zeta}}{\sqrt{\zeta_1^2 + \cdots + \zeta_p^2}},
\]
where $\zeta_1, \ldots, \zeta_p$ are independent $\N(0,1)$ random variables. Using this representation, we get
\begin{equation}\label{eq:p_small_se}
\begin{aligned}
&\P\left( \bm{\eta}^\top (\bD^\top \bD)^{-1} \bm{\eta} \ge \frac{n}{n - p} + c \Big | \bD\right)\\
&= \P\left( \bm{\zeta}^\top (\bD^\top \bD)^{-1} \bm{\zeta} \ge \|\bm{\zeta}\|^2 \left(\frac{n}{n - p} + c\right) \Bigg | \bD\right)\\
&\le \P\left( \bm{\zeta}^\top (\bD^\top \bD)^{-1} \bm{\zeta} \ge (p - p^{0.75}) \left(\frac{n}{n - p} + c\right) \Bigg | \bD\right) + \P(\|\bm\zeta\|^2 \le p - p^{0.75}).
\end{aligned}
\end{equation}
Setting $t = \sqrt{p}$, Assumption \ref{ass:working} ensures that the following inequality must eventually hold:
\[
(p - p^{0.75}) \left(\frac{n}{n - p} + c\right) \ge p\left(\frac{n}{n - p} + \frac{c}{2} \right) + (2\sqrt{pt} + 2t) \times \left( \frac{\delta}{(\sqrt{\delta} - 1)^2} + c \right).
\]
Plugging the display above into \eqref{eq:p_small_se} gives
\[
\begin{aligned}
&\P\left( \bm{\eta}^\top (\bD^\top \bD)^{-1} \bm{\eta} \ge \frac{n}{n - p} + c \Big | \bD\right)\\
& \le \P\left( \bm{\zeta}^\top (\bD^\top \bD)^{-1} \bm{\zeta} \ge p\left(\frac{n}{n - p} + \frac{c}{2} \right) + (2\sqrt{pt} + 2t) \times \left( \frac{\delta}{(\sqrt{\delta} - 1)^2} + c \right) \Bigg | \bD\right) + \P(\|\bm\zeta\|^2 \le p - p^{0.75}).
\end{aligned}
\]
Next, denote by
\[
\A = \left\{ \tr((\bD^\top \bD)^{-1}) \ge p\left(\frac{n}{n - p} + \frac{c}{2} \right) \text{ or } \|(\bD^\top \bD)^{-1}\| > \frac{\delta}{(\sqrt{\delta} - 1)^2} + c \right\}
\]
and write $\mathbf{1}_{\A}$ for the indicator function of $\A$ 
defined as
\[
{\displaystyle \mathbf{1} _{\A} :={\begin{cases}1&{\text{if the event } \A \text{ holds}},\\0&{\text{otherwise}}.\end{cases}}}
\]
Then, it is easy to see that
\[
\begin{aligned}
&\P\left( \bm{\zeta}^\top (\bD^\top \bD)^{-1} \bm{\zeta} \ge p\left(\frac{n}{n - p} + \frac{c}{2} \right) + (2\sqrt{pt} + 2t) \times \left( \frac{\delta}{(\sqrt{\delta} - 1)^2} + c \right) \Bigg | \bD\right)\\
&\le \mathbf{1}_{\A} + \P\left( \bm{\zeta}^\top (\bD^\top \bD)^{-1} \bm{\zeta} \ge \tr((\bD^\top \bD)^{-1}) + (2\sqrt{pt} + 2t)\|(\bD^\top \bD)^{-1}\| \Bigg | \bD\right)\\
&\le \mathbf{1}_{\A}  + \e^{-t},
\end{aligned}
\]
where the last step follows from a Gaussian concentration inequality (see, for example, \citet{hsu2012tail}). As a result, it suffices to show that
\begin{equation}\nonumber
\E \min \left\{ p\mathbf{1}_{\A} + p\e^{-t} + p \P(\|\bm\zeta\|^2 \le p - p^{0.75}), 1 \right\} \rightarrow 0.
\end{equation}

To this end, first note that from
\[
\frac1{(1 - \sqrt{p/n})^2} + c < \frac{\delta}{(\sqrt{\delta} - 1)^2} + c
\]
we get
\begin{equation}\label{eq:d_spec}
\P\left(\|(\bD^\top \bD)^{-1}\| > \frac{\delta}{(\sqrt{\delta} - 1)^2} + c \right) \le \P\left(\|(\bD^\top \bD)^{-1}\| > \frac1{(1 - \sqrt{p/n})^2} + c\right) \goto 0, 
\end{equation}
where we use the fact that the smallest singular value of the Wishart matrix $\bX^\top \bX$ is concentrated at $(1 - \sqrt{p/n})^2$ with probability tending to one (see, for example, \citet{vershynin}). Second, it is known that
\[
\tr((\bD^\top \bD)^{-1}) = (1 + o_{\P}(1)) \frac{np}{n - p},
\]
which implies
\begin{equation}\label{eq:trace}
\P\left( \tr((\bD^\top \bD)^{-1}) \ge p\left(\frac{n}{n - p} + \frac{c}{2} \right) \right) \goto 0.
\end{equation}
Taking \eqref{eq:d_spec} and \eqref{eq:trace} together, we get
\begin{equation}\label{eq:a_zero}
\P(\A) \goto 0.
\end{equation}
Recognizing that $t = \sqrt{p}$, we get $p\e^{-t} \goto 0$. In addition, $p \P(\|\bm\zeta\|^2 \le p - p^{0.75}) \goto 0$. Hence, from \eqref{eq:a_zero} we have
\[
\begin{aligned}
\min \left\{ p \mathbf{1}_{\A}  + p\e^{-t} + p \P(\|\bm\zeta\|^2 \le p - p^{0.75}), 1 \right\} \rightarrow 0
\end{aligned}
\]
in probability. In addition, note that
\[
\min \left\{ p \mathbf{1}_{\A} + p\e^{-t} + p \P(\|\bm\zeta\|^2 \le p - p^{0.75}), 1 \right\}
\]
is upper bounded by $1$, which is an integrable function. Hence, by the dominated convergence theorem, we get
\[
\E \min \left\{ p\mathbf{1}_{\A} + p\e^{-t} + p \P(\|\bm\zeta\|^2 \le p - p^{0.75}), 1 \right\} \rightarrow 0,
\]
as desired.
\end{proof}



\end{document}